\documentclass[12pt,a4paper]{article}
\usepackage{amsmath,amsfonts,amssymb,amsthm,amscd}
\usepackage{verbatim}

\usepackage[paper=a4paper,left=30mm,right=20mm,top=25mm,bottom=30mm]{geometry}

\newtheorem{thm}{Theorem}
\newtheorem{lem}[thm]{Lemma}
\newtheorem{prop}[thm]{Proposition}
\newtheorem{defn}[thm]{Definition}

\newtheorem{notation}[thm]{Notation}

\begin{document}

\newfont{\goth}{eufm10 scaled \magstep1}
\def\ga{\mbox{\goth a}}
\def\gp{\mbox{\goth p}}
\def\gc{\mbox{\goth c}}
\def\gg{\mbox{\goth g}}
\def\ge{\mbox{\goth e}}
\def\gh{{\mbox{\goth h}}}
\def\gk{\mbox{\goth k}}
\def\gl{\mbox{\goth l}}
\def\gf{\mbox{\goth f}}
\def\gm{\mbox{\goth m}}
\def\gn{\mbox{\goth n}}
\def\gq{\mbox{\goth q}}
\def\gr{\mbox{\goth r}}
\def\gs{\mbox{\goth s}}
\def\gt{\mbox{\goth t}}
\def\gu{\mbox{\goth u}}
\def\gw{\mbox{\goth w}}
\def\gsl{\mbox{\goth sl}}
\def\gsp{\mbox{\goth sp}}
\def\r{\mbox{\goth r}}
\def\gso{\mbox{\goth so}}
\def\gsu{\mbox{\goth su}}
\newcommand\Id{\mathrm{Id}}
\newcommand\Ad{\mathrm{Ad}}
\newcommand\ad{\mathrm{ad}}
\def\ra{\rightarrow}
\def\ggl{\mbox{\goth gl}}

\author{Dmitri Alekseevsky and Liana David}

\title{Prolongation of Tanaka structures: an alternative approach}

\maketitle

Dmitri V. Alekseevsky: {\it Institute for Information Transmission
Problems},\\ B. Karetny per. 19, 127051, Moscow (Russia) and\\
{\it  University of Hradec Kr\'{a}lov\'{e}},\\
Rokitanskeho 62, Hradec Kr\'{a}lov\'{e}  50003 (Czech Republic),
dalekseevsky@iitp.ru\\

Liana  David\footnote{Corresponding Author}: {\it ``Simion Stoilow'' Institute of the Romanian
Academy},\\
Research Unit 7, P.O. Box 1-764, Bucharest, (Romania); tel. 0040-21-31965065, fax. 0040-21-3196505,
e-mail: liana.david@imar.ro\\

\begin{center}

{\it Dedicated to the 150th anniversary of the Romanian Academy}
\end{center}

{\bf Abstract:}
The classical theory of prolongation of $G$-structures was generalized by N. Tanaka
to  a  wide   class  of  geometric structures  (Tanaka  structures), which are defined on a non-holonomic distribution.
Examples of Tanaka structures include  subriemannian,  subconformal, CR  structures, structures  associated to second order differential  equations and structures defined
by gradings of Lie  algebras (in the framework of parabolic geometries).
Tanaka's  prolongation procedure associates  to a Tanaka  structure  of  finite  order  a manifold  with an absolute parallelism. It is a very fruitful method
for the  description of local invariants, investigation  of  the   automorphism  group  and
equivalence problem. In  this paper,
we  develop  an  alternative  constructive  approach  for  Tanaka's prolongation procedure, based  on  the theory  of quasi-gradations
of filtered vector spaces, $G$-structures and
their torsion functions.\\

{\bf Key words:} $G$-structures, Tanaka structures, prolongations, automorphism groups, quasi-gradations,
torsion functions.\\

{\bf MSC Classification:} 53C10, 53C15, 53B15.\\

D.V.A. carried this work at  IITP  and  is supported by an RNF grant (project
n.14-50-00150).
L.D. carried this work during a
Humboldt Research Fellowship at the University of Hamburg.
She thanks  University of Hamburg for hospitality  and
the Alexander von Humboldt-Stiftung for financial support.
Partial financial support from  the CNCS-UEFISCDU grant
PN-II-ID-PCE-2011-3-0362 is also acknowledged.

\newpage

\section{Introduction}

Recall   that a $G$-structure of  an $n$-dimensional manifold  $M$  is  a principal    subbundle  $\pi_G : P_G \to M $ of  the   frame  bundle
of $M$ with structure  group $G \subset GL(V),\, V = \mathbb{R}^n$. Any  tensor  field which  is infinitesimally  homogeneous,
i.e. whose  value   at  any point has  the  same normal  form  with  respect  to   some "admissible" frame,
is identified  with  a $G$-structure,  whose total space  $P_G$ is  the set of all such admissible  frames. \\
The prolongation of $G$-structures (see e.g. \cite{sternberg},
Chapter VII) is a powerful  method  in differential geometry
which associates  to any $G$-structure $ \pi_G : P_G \to M $
of finite order  a new  manifold $ P= P(\pi_G )$ (the  full prolongation), with an
absolute parallelism (i.e. an $\{e\}$-structure),
with the important property that the group of automorphisms $\mathrm{Aut}(P,\{ e\} )$ of
$(P, \{ e\})$  is isomorphic to
the group of automorphisms
$\mathrm{Aut}(\pi_{G})$  of $\pi_{G}$. The absolute paralellism $(P, \{ e\})$
provides local invariants for
$\pi_{G}$  (see \cite{sternberg}, Theorem 4.1 of Chapter
VII). Owing to Kobayashi's theorem (see \cite{kobayashi},
Theorem 3.2 of Chapter 0),   $\mathrm{Aut}(\pi_{G}) \simeq \mathrm{Aut}(P,\{ e\} )$
are Lie  groups of dimension less or equal to the dimension of $P$.\

The   full  prolongation $P$
of $\pi_{G}: P_{G} \ra M$ is  defined by consecutive applications  of  the first prolongation. We briefly recall its construction.
It is  based on the observation  that the bundle $j^1(\pi_G) :  J^1 P_{G} = \mathrm{Hor}(P_{G}) \to  P_{G}$ of  1-jets  of  sections
of $\pi_{G}$ (i.e. horizontal  subspaces  of $TP_G$)
is a $G$-structure with structure group $$ G^1 = \mathrm{id} + \mathrm{Hom}(V, \mathfrak{g}) =
 \{
 \begin{pmatrix}
  \mathrm{id} & 0 \\
  A_{.} & \mathrm{id}\\
\end{pmatrix}
 ,\,\,  A_{.} \in \mathrm{Hom}(V, \mathfrak{g}) \}, $$
which is a  commutative subgroup  of  $GL(V + \mathfrak{g})$.
Using  the torsion functions of $j^{1} (\pi_{G})$, one  can  reduce
the $G$-structure $j^{1} (\pi_{G})$
to a $G$-structure $\pi^{(1)}_{G}: P_G^{(1)} \to  P_G$ whose structure group $G^{(1)}$
is  the  Lie subgroup of $G^{1}$ generated  by the Lie
subalgebra  $\mathfrak{g}^{(1)} = \mathrm{Hom}(V, \mathfrak{g}) \cap (V \otimes S^2V^*)  \subset \mathfrak{gl}(V + \mathfrak{g}). $
The $G$-structure $\pi^{(1)}_{G}$ is called the first prolongation of $\pi_{G}.$
If the  $k$-th iterated  prolongation $\mathfrak{g}^{(k)} :=( \mathfrak{g}^{(k-1)})^{(1)}$ of the Lie  algebra $\mathfrak{g}= \mathrm{Lie} (G)$  vanishes,
then  $G$ is  called  of   finite  order  and  the  $k$-th iterated $P^{(k)}_{G}$
first prolongation of $P_G$ defines  an absolute parallelism on  the
full prolongation    $P:= P_G^{(k-1)} $.\

While the prolongation procedure works effectively for
$G$-structures of finite order (e.g. conformal or quaternionic structures),
there are other important geometric structures (e.g.
CR-structures  and other  structures   defined on a non-integrable  distribution),
which cannot be treated effectively by this method.
To overcome this difficulty, in 1970  Tanaka \cite{tanaka}
generalized the prolongation of $G$-structures to  a larger class
of geometric structures, called Tanaka structures in
\cite{dmitri-spiro} and infinitesimal flag structures in
\cite{cap-slovak} (see Definition \ref{def-tanaka-gen}).
Examples  of Tanaka structures include CR-structures, subriemannian
and subconformal structures.  Tanaka's prolongation
procedure received much attention in the mathematical literature.
There are many approaches for the Tanaka prolongation  under  different  assumptions,
see \cite{morimoto2, caps, dmitri-spiro, zelenko, cap-slovak}.
Our approach is a  developing and a  detalization of  the approach from \cite{dmitri-spiro}, where
the  first  step  of the Tanaka prolongation was  explained  in  detail, but the
other  steps  were  only stated without proofs. To  prove  the  iterative   construction,  one  has  to
check many extra  conditions,  and this will be carefully  done  in this paper.
Our approach  is  close to the approach of
I. Zelenko \cite{zelenko}. The main  difference is   that  we  develop   and   systematically use the  theory of
quasi-gradations of   filtered vector  spaces. Together  with the well-known theory of Tanaka prolongations
of non-positively graded Lie  algebras and the   torsion functions of $G$-structures, this provides
a conceptual and  simple  description
of each step of the prolongation procedure: the   principal  bundle
$\bar{\pi}^{(n)}: \bar{P}^{(n)}\ra \bar{P}^{(n-1)}$ which relates the $n$ and $(n-1)$-prolongations of a given Tanaka structure
is canonically isomorphic to a subbundle of the principal bundle of
$(n+1)$-quasi-gradations of  $T\bar{P}^{(n-1)}$
and is obtained as
the quotient of a $G$-structure of $\bar{P}^{(n-1)}$, with structure
group $G^{n} GL_{n+1} (\gm_{n-1})$, with  suitable properties of the
torsion function. These statements are explained in detail in
Theorem \ref{sf0}. In Theorem \ref{sf} we state the final result of the Tanaka prolongation procedure,
which reduces the local classification of Tanaka structures of finite order to the well understood local classification of
absolute parallelisms.  This requires
the construction of a canonical frame on a prolongation of suitable order and
a careful analysis of the behaviour of the automorphisms of a Tanaka structure under the prolongation procedure.
We do this in Propositions \ref{ad-prop} and Proposition \ref{ad-prop-1}.
In the remaining part of the
introduction we present the structure of the paper.\\

{\bf Structure of the paper.} Section \ref{algebraic-sect} is
mainly intended to fix notation. Our original
contribution in this section is the theory of quasi-gradations of
filtred vector spaces, which is developed in Subsections \ref{quasi-gradations} and
\ref{lifts-sect}. Besides, we  recall the definition of the Tanaka
prolongation of a non-positively graded Lie algebra \cite{tanaka},
the basic facts we need  from the theory of $G$-structures (see e.g. \cite{sternberg}) and the definition of
Tanaka structure \cite{dmitri-spiro}.

In Section \ref{statements} we state our main results from this
paper, namely Theorems \ref{sf0} and \ref{sf}. All notions used
in these statements are defined in the previous section.

The remaining sections are devoted to the proofs of
Theorems \ref{sf0} and \ref{sf}. Let $(\mathcal D_{i} , \pi_{G} :
P_{G} \ra M)$ be a Tanaka $G$-structure of type $\gm  = \sum_{i=-k}^{-1} \gm^{i}.$ Basically,
the proof of Theorem \ref{sf0} is divided into two main parts: in
a first stage, in Section \ref{1st-prol-sect} we construct the
starting projection $\bar{\pi}^{(1)} : \bar{P}^{(1)} \ra P =
P_{G}$ of the sequence of projections from Theorem \ref{sf0} (also
called the first prolongation of the Tanaka structure $(\mathcal
D_{i} , \pi_{G})$).
For this, we remark that $P$ has a canonical Tanaka $\{ e\}$-structure of type $\gm_{0} = \gm + \gg^{0}$
(where $\gg^{0} =
\mathrm{Lie}(G)$) and we define a  $G$-structure
$\pi^{1} : P^{1} \ra P$ as the set of all
adapted gradations of $TP$, or, equivalently,  the set of all
frames of $TP$ which lift the canonical graded frames of the Tanaka $\{ e\}$-structure of $P$
(see Proposition \ref{hat} and Definition \ref{pi-1}).
Using the torsion, we reduce $\pi^{1}$ to a subbundle $\tilde{\pi}^{1}: \tilde{P}^{1} \ra P$,
with structure group $G^{1}GL_{2} (\gm_{1})$ and we define
$\bar{\pi}^{(1)} : \bar{P}^{(1)} \ra P =
P_{G}$ to be the quotient of $\tilde{\pi}^{1}$  by the normal subgroup
$GL_{2}(\gm_{1})$  (see
Definition \ref{first-prol-def}). To a large extent
(except Subsection \ref{act-sub-1})
this material
is a rewriting of the construction from \cite{dmitri-spiro}, using frames instead of coframes (which
are more suitable for the higher steps of the prolongation).
It is also the simplest part of the prolongation procedure. We skip its details in this introduction and we describe directly
the higher steps of the prolongation, where our new approach using quasi-gradations
plays a crucial role. Therefore, suppose that the projections
$\bar{\pi}^{(i)}:\bar{P}^{(i)} \ra \bar{P}^{(i-1)}$ ($i\leq n$)
from Theorem  \ref{sf0} are given. We aim to define
$\bar{\pi}^{(n+1)}:\bar{P}^{(n+1)} \ra \bar{P}^{(n)}.$

In Section \ref{pi-n} we define $P^{n+1} \subset \mathrm{Gr}(
T\bar{P}^{(n)})$ as the set of all adapted gradations of
$T_{\bar{H}^{n}} \bar{P}^{(n)}$ (for any $\bar{H}^{n}\in
\bar{P}^{(n)}$), whose projection to
$T_{\bar{H}^{n-1}}\bar{P}^{(n-1)}$ is compatible with the
quasi-gradation $\bar{H}^{n} \in \mathrm{Gr}_{n+1}
(T_{\bar{H}^{n-1}}\bar{P}^{(n-1)})$ (see
Definition \ref{pi-n-t})
and we  show that the natural map ${\pi}^{n+1}: P^{n+1} \ra
\bar{P}^{(n)}$  is a $G$-structure, with structure group $\mathrm{Id} + \mathfrak{gl}_{n+1} (\gm_{n}) +
\mathrm{Hom}(\sum_{i=0}^{n-1} \gg^{i}, \gm_{n})$ (see Proposition
\ref{pi-2}).

The definition of $\bar{\pi}^{(n+1)}$  requires
a careful analysis of the torsion functions of the $G$-structure
${\pi}^{n+1}.$ This is done in Sections \ref{torsion-g-n}
and \ref{variation-g-n}. In Section \ref{torsion-g-n}  we consider
an arbitrary connection $\rho$ on the $G$-structure $\pi^{n+1}:
P^{n+1} \ra \bar{P}^{(n)}$ and we
study the component   $t^{\rho}: P^{n+1} \ra \mathrm{Hom} ( (\gm^{-1} + \gg^{n})\wedge \gm_{n},\gm_{n}) $
of its torsion function
(see Theorem
\ref{main-t}).
The proof of Theorem \ref{main-t} is divided into three parts, according to the
decomposition of  $\mathrm{Hom} ( (\gm^{-1} + \gg^{n})\wedge \gm_{n},\gm_{n})$
into the subspaces
$\mathrm{Hom}( \gg^{n}\wedge\gm_{n}, \gm_{n})$,
$\mathrm{Hom} (\gm^{-1}\wedge \gm ,
\gm_{n})$ and
$\mathrm{Hom}(\gm^{-1}\wedge \sum_{i=0}^{n-1}
\gg^{i},\gm_{n})$.
In Section \ref{action-g-n}
we  define an action of $G^{n} GL_{n+1}
(\gm_{n-1})$ on ${P}^{n+1}$
(see Proposition \ref{cheie})
which is used to treat
the  $\mathrm{Hom}( \gg^{n}\wedge\gm_{n}, \gm_{n})$-valued component
of $t^{\rho}$ (see Proposition \ref{part3}). The properties of the $\mathrm{Hom} (\gm^{-1}\wedge \gm ,
\gm_{n})$-valued component of $t^{\rho}$ are consequences of the fact that
the canonical
graded frames of the Tanaka $\{ e\}$-structure on $\bar{P}^{(n)}$
are Lie algebra isomorphisms when restricted to $\gm$ (see Proposition \ref{part1}). The properties of the
remaining $\mathrm{Hom}(\gm^{-1}\wedge \sum_{i=0}^{n-1}
\gg^{i},\gm_{n})$-valued component  of $t^{\rho}$
are inherited from
the properties of the torsion
function of the $G$-structure $\tilde{\pi}^{n}: \tilde{P}^{n} \ra
\bar{P}^{(n-1)}$ (see Proposition \ref{part2}).

In Section \ref{variation-g-n} we determine the homogeneous
components of  $t^{\rho}$  which are independent of the
connection $\rho$ and we define and study  the $(n+1)$-torsion
$\bar{t}^{(n+1)} : P^{n+1} \ra \mathrm{Tor}^{n+1} (\gm_{n})$
of the Tanaka structure $(\mathcal D_{i}, \pi_{G})$ (see Definition \ref{torsion-n-def} and Theorem \ref{modificare-tors}).

With the material from the previous sections,
in Section \ref{prol-sect} we finally define  the $G$-structure $\tilde{\pi}^{n+1} : \tilde{P}^{n+1} \ra \bar{P}^{(n)}$
and the principal bundle  $\bar{\pi}^{(n+1)}: \bar{P}^{(n+1)} \ra \bar{P}^{(n)}$ we are
looking for. Let $W^{n+1}$ be a complement of $\mathrm{Im}(\partial^{(n+1)})$ in the space
of torsions $\mathrm{Tor}^{n+1}(\gm_{n})$ (see Theorem \ref{modificare-tors} for the definition of
the map $\partial^{(n+1)}$). The $G$-structure
$\tilde{\pi}^{n+1}$ is
the restriction of $\pi^{n+1}$ to $\tilde{P}^{n+1} = (\bar{t}^{(n+1)})^{-1}(W^{n+1})$
and has structure group $G^{n+1} GL_{n+2}
(\gm_{n})$ (see Proposition \ref{tilde-n}). The bundle
$\bar{\pi}^{(n+1)}: \bar{P}^{(n+1)}\ra \bar{P}^{(n)}$ is defined
as the quotient of $\tilde{\pi}^{n+1}$  by the normal subgroup
$GL_{n+2}(\gm_{n})\subset G^{n+1}GL_{n+2} (\gm_{n})$ and satisfies the properties from
Theorem \ref{sf0} (see Proposition  \ref{prol-n-1}). This concludes the proof of Theorem
\ref{sf0}.

In Section \ref{prelungire-tan} we prove  Theorem \ref{sf}.
The construction of the canonical frame  $F^{\mathrm{can}}$
on $\bar{P}^{(\bar{l})}$ (or on any $\bar{P}^{(\bar{l}^{\prime})}$, for $\bar{l}^{\prime}\geq \bar{l}$),
required by Theorem \ref{sf}, is done in Proposition
\ref{ad-prop}. In Proposition \ref{ad-prop-1} we show that the automorphism group $\mathrm{Aut} ({\mathcal D}_{i}, \pi_{G})$ of a Tanaka structure $(\mathcal D_{i}, \pi_{G})$
(not necessarily of finite order) is isomorphic to the automorphism group of any of  the associated $G$-structures $\tilde{\pi}^{n}: \tilde{P}^{n} \ra \bar{P}^{(n-1)}$
($n\geq 1$).
When  $(\mathcal D_{i}, \pi_{G})$ is of finite
order
$\bar{l}$, the $G$-structure $\tilde{\pi}^{\bar{l}^{\prime}+1}: \tilde{P}^{\bar{l}^{\prime} +1} \ra \bar{P}^{(\bar{l}^{\prime})}$
is an absolute parallelism for large enough $\bar{l}^{\prime}$, which coincides with the
canonical frame of $\bar{P}^{(\bar{l}^{\prime})}$  (see Proposition \ref{ad-prop-2}).
This fact, combined with Proposition \ref{ad-prop-1} and Kobayashi's theorem mentioned above,  completes the proof of Theorem \ref{sf}.\\

\section{Preliminary material}\label{algebraic-sect}

\subsection{Quasi-gradations of filtred vector
spaces}\label{quasi-gradations}

Let $V= V_{-k} \supset V_{-k+1} \supset \dots \supset V_{l}$ be a
decreasing filtration of a finite dimensional vector space $V$ by
subspaces $V_{i}.$ We define $V_{j} = \{ 0\}$ for $j
> l$ and $V_{j} = V$ for $j < - k.$

\begin{defn}\label{def-quasi}
i) A gradation
$H = \{ H^{i},\ -k\leq i\leq l\}$ of $V$ is called {\bf adapted} (to the filtration $\{ V_{i}\}$) if
$V_{i} = H^{i}+ H^{i+1}+ \cdots + H^{l}$, for any $- k \leq i \leq l.$

\medskip

ii) A {\bf quasi-gradation of degree $m\geq 1$} (or shortly,
{\bf $m$-quasi-gradation}) of $V$ is a system of subspaces
$\bar{H}=\{ \bar{H}^{i},\ -k\leq i \leq  l\} $ such that, for any $-k\leq i\leq l$,
$$
a) V_{i} = \bar{H}^{i}+ V_{i+1},\quad b) \bar{H}^{i}\cap V_{i+1}
= V_{i+m}.
$$
\end{defn}

We denote by $\mathrm{Gr}(V)$ and $\mathrm{Gr}_{m}(V)$ the set of all adapted gradations, respectively
the set of all $m$-quasi-gradations of $V$. Remark that $\mathrm{Gr}_{m}(V) =
\mathrm{Gr}(V)$ for any $m\geq k+l+1$. For any $1\leq m \leq  p$,
we define
$$
\Pi_{p}^{m}: \mathrm{Gr}_{p}(V) \rightarrow \mathrm{Gr}_{m}(V),\
\Pi_{p}^{m} (\{ \bar{H}^{i}\} ) := \{ \bar{H}^{i}+ V_{i+m}\} .
$$
In particular, there is a natural map
\begin{equation}\label{pi-m}
\Pi^{m}:\mathrm{Gr} (V) \rightarrow \mathrm{Gr}_{m}(V),\
\Pi^{m} (\{ H^{i}\} ):=  \{ H^{i}+ V_{i+m}\} .
\end{equation}

\begin{defn}\label{def-compatible}
Any adapted gradation of $V$ which belongs to
$\mathrm{Gr}_{\bar{H}} (V):= (\Pi^{m})^{-1} (\bar{H})$ is called
{\bf compatible with the quasi-gradation $\bar{H}\in \mathrm{Gr}_{m} (V)$}.\end{defn}

Let $\mathrm{gr} (V) := \sum_{i=-k}^{l}  \mathrm{gr}^{i}(V)$, where
$\mathrm{gr}^{i}(V):= V_i/V_{i+1}$, be the graded vector space
associated to $V$. More generally, for any $m\geq 1$, let $\mathrm{gr}_{(m)} (V)
:=\sum_{i= -k}^{l} \mathrm{gr}^{i}_{(m)}(V)$, where
$\mathrm{gr}^{i}_{(m)}(V):= V_{i}/ V_{i+m}$. We denote by
$$
\mathrm{gr}^{i}: V_i \to \mathrm{gr}^{i}(V),\ \mathrm{pr}^{i}_{(m)} : V_{i}
\ra \mathrm{gr}^{i}_{(m)}(V),\ \mathrm{gr}^{i}_{(m)}:
\mathrm{gr}^{i}_{(m)} (V) \ra \mathrm{gr}^{i}(V)
$$
the natural projections. Remark that $\mathrm{gr}^{i} =
\mathrm{pr}^{i}_{(1)} $ and $\mathrm{pr}^{i}_{(1)} =
 \mathrm{gr}^{i}_{(m)}$ for $m\geq k+l+1.$
Any adapted gradation $H = \{ H^{i}\}$ defines injective maps $\widehat{H}^{i}:
\mathrm{gr}^{i}(V) \ra V_{i}$, with image $H^{i}\subset V_{i}$
(from the direct sum decompositions $V_{i} = V_{i+1} + H^{i}$). 
The next
proposition generalizes this statement to quasi-gradations.

\begin{prop}\label{functii} i) There is a one to one correspondence between the space
$\mathrm{Gr}_{m}(V)$ of $m$-quasi-gradations $\bar{H} = \{
\bar{H}^{i}\} $ and the space of  maps
${f} = ({f}^{i}): \mathrm{gr}(V) \ra \mathrm{gr}_{(m)}(V)$
where
\begin{equation}\label{cond}
{f}^{i} : \mathrm{gr}^{i}(V) \ra \mathrm{gr}^{i}_{(m)}(V),\
\mathrm{gr}^{i}_{(m)} \circ {f}^{i} = \mathrm{Id}_{\mathrm{gr}^{i}(V)},\quad -k\leq i\leq l.
\end{equation}
More precisely,  any $\bar{H}\in \mathrm{Gr}_{m}(V)$ defines a
map $\widehat{\bar{H}}= ( \widehat{\bar{H}}^{i} ):
\mathrm{gr}(V) \ra \mathrm{gr}_{(m)}(V)$ which satisfies
(\ref{cond}) and $\widehat{\bar{H}}^{i} : \mathrm{gr}^{i}(V) \ra
\mathrm{gr}^{i}_{(m)}(V)$ has image $\bar{H}^{i}/ V_{i+m}\subset
\mathrm{gr}^{i}_{(m)}(V)$. Conversely, any map ${f} = (f^{i}) :\mathrm{gr}(V) \ra \mathrm{gr}_{(m)}(V)$
as in  (\ref{cond})
defines  $\bar{H}= \{ \bar{H}^{i}\}\in \mathrm{Gr}_{m}(V)$  by
\begin{equation}\label{quasi-asoc}
\bar{H}^{i}:= (\mathrm{pr}^{i}_{(m)})^{-1} \mathrm{Im} (f^{i}),\quad  -k\leq i\leq l
\end{equation}
and $f=\widehat{\bar{H}}.$

\medskip

ii) A gradation $H$ is compatible with an $m$-quasi-gradation $\bar{H}$
if and only if
\begin{equation}\label{expr-fc}
\mathrm{pr}^{i}_{(m)}\circ \widehat{H}^{i} =
\widehat{\bar{H}}^{i},\quad   -k\leq i\leq l.
\end{equation}
\end{prop}

\begin{proof}
The proof is straightforward and we omit details. We only define the map
$\widehat{\bar{H}}$ associated to the quasi-gradation $\bar{H}\in \mathrm{Gr}_{m}(V)$, and this
is done as for gradations.
Namely, from Definition
\ref{def-quasi}, $V_{i}/ V_{i+m} = \bar{H}^{i}/ V_{i+m} + V_{i+1}/
V_{i+m}$ (direct sum decomposition). This induces an isomorphism
between $\mathrm{gr}^{i} (V) = (V_{i}/V_{i+m}) /
(V_{i+1}/V_{i+m})$ and $\bar{H}^{i}/ V_{i+m}\subset
\mathrm{gr}^{i}_{(m)}(V)= V_{i}/ V_{i+m}$, which gives the
required map $\widehat{\bar{H}}^{i}.$
Alternatively, $\widehat{\bar{H}}^{i}$ associates to $[y]\in \mathrm{gr}^{i}(V)$ the unique
$[z]\in  \mathrm{pr}^{i}_{(m)} (\bar{H}^{i})\subset \mathrm{gr}^{i}_{(m)} (V)$, such that $\mathrm{gr}^{i}_{(m)} ([z]) = [y]$.
\end{proof}

\subsection{Lifts and quasi-gradations}\label{lifts-sect}

Let $\gm = \sum_{i} \gm^{i}$ be a graded vector space, $V$ a filtered vector space and
$u: \gm \ra \mathrm{gr}(V)$  a graded vector space
isomorphism. Since $\gm$ is  graded, it is filtered in a natural way by the subspaces
$\gm_{i}:=\sum_{j\geq i} \gm^{j}$.

\begin{defn} A {\bf lift} of $u$ is a  filtration preserving isomorphism $F:
\gm\ra V$ which satisfies $\mathrm{gr}^{i}\circ F\vert_{\gm^{i}}=
u\vert_{\gm^{i}} $, for any $i$. More generally, an {\bf $m$-lift}
$(m\geq 1$)  is a map ${F} = ( F^{i}) : \gm \ra
\mathrm{gr}_{(m)}(V)$, where $F^{i}: \gm^{i} \ra
\mathrm{gr}^{i}_{(m)} (V)$ are such that
$\mathrm{gr}^{i}_{(m)} \circ F^{i} = u\vert_{\gm^{i}}$, for any
$i$.
\end{defn}

We remark that $F$ is a lift of $u$ if and only if it is
filtration preserving and $\mathrm{gr}^{i}\circ F\vert_{\gm_{i}}=
u\circ \pi_{\gm^i}\vert_{\gm_{i}} $, for any $i$.
(We always denote by $\pi_{\gm^{i}}:\gm \ra \gm^{i}$
the natural projection onto the degree $i$-component
$\gm^{i}$ of a graded vector space $\gm$).
The next
theorem generalizes  Lemma 7.1 of \cite{dmitri-spiro}.

\begin{thm}\label{g-f-vspace} There is a one to one correspondence between the space of $m$-quasi-gradations
of $V$ and the space of  $m$-lifts of $u$. More precisely, any  $m$-quasi-gradation $\bar{H}$
defines an $m$-lift, by $F_{\bar{H}}^{i} := \widehat{\bar{H}}^{i}
\circ u\vert_{\gm^{i}}.$ Conversely, any $m$-lift $F= (F^{i})$
defines an $m$-quasi-gradation $\bar{H}^{i} :=
(\mathrm{pr}^{i}_{(m)})^{-1} F^{i} (\gm^{i})$ and $F=F_{\bar{H}}.$
\end{thm}

\begin{proof}
Let $\bar{H}\in \mathrm{Gr}_{m}(V)$.
From the definitions of
$F_{\bar{H}}^{i}$ and
$\widehat{\bar{H}}^{i}$,
$\mathrm{gr}^{i}_{(m)} \circ  F_{\bar{H}}^{i} =
\mathrm{gr}^{i}_{(m)} \circ \widehat{\bar{H}}^{i} \circ
u\vert_{\gm^{i}} = u\vert_{\gm^{i}}$, i.e.
$F_{\bar{H}}$ is an $m$-lift. Conversely, if $F$ is an
$m$-lift, then $F \circ u^{-1}: \mathrm{gr}(V) \ra
\mathrm{gr}_{(m)}(V)$ satisfies the properties from Proposition
\ref{functii}. We deduce that
$\bar{H}:= \{ \bar{H}^{i}\}$ where
$$
\bar{H}^{i}=
(\mathrm{pr}^{i}_{(m)})^{-1} \mathrm{Im} (F\circ u)^{i} =
(\mathrm{pr}^{i}_{(m)})^{-1} F^{i}(\gm^{i})
$$
is an $m$-quasi-gradation.
It remains to prove that $F=F_{\bar{H}} $. For this, let $x\in \gm^{i}$.
Since $\bar{H}^{i} = (\mathrm{pr}^{i}_{(m)})^{-1} F^{i}(\gm^{i})$,
$F^{i}(x) \in \mathrm{pr}^{i}_{(m)} (\bar{H}^{i})$. Since
$\mathrm{gr}^{i}_{(m)} \circ F^{i}(x) = u(x)$, we obtain
$F_{\bar{H}}^{i}(x) = \widehat{\bar{H}^{i}} (u(x)) = F^{i}(x)$,
as needed (the second equality follows from   the proof of Proposition \ref{functii}, by taking  $[y] = u(x)$ and $[z]= F^{i}(x)$).
\end{proof}

In view of the above theorem, we identify the space
$\mathrm{Gr}_{m}(V)$ of $m$-quasi-gradations with the space of
$m$-lifts of $u$. To avoid confusion, lifts of $u$ will be denoted
by $F_{H}$ and $m$-lifts by $F_{\bar{H}}$.  The map (\ref{pi-m}),
in terms of $m$-lifts, is
\begin{equation}\label{map-pi}
\Pi^{m}:\mathrm{Gr}(V) \ra \mathrm{Gr}_{m}(V),\ F_{H}=(F^{i}_{H})
\mapsto  F_{\bar{H}}= (F_{\bar{H}}^{i}:=\mathrm{pr}^{i}_{(m)} \circ
F^{i}_{H}).
\end{equation}

We end this subsection by discussing group actions on the space of quasi-gradations.
For this, we need to introduce new notation, which will be used also later in the paper.
Recall  that
if $U:= \sum_{i} U^{i}$ and $W:= \sum_{j} W^{j}$ are graded vector spaces, then
$U\wedge W:= \sum_{i} (U\wedge W)^{i}$ and
$\mathrm{Hom}(U, W) = \sum_{i}\mathrm{Hom}^{i}(U, W)$ are graded as well,
where $(U\wedge W)^{i} := \sum_{j+r=i}U^{j}\wedge W^{r}$
and $\mathrm{Hom}^{i}(U,W)
:= \sum_{j}\mathrm{Hom}(U^{j}, W^{j+i}).$
For any $A\in \mathrm{Hom}(U, W)$, we  denote by $A^{i}\in \mathrm{Hom}^{i}(U,W)$
its degree $i$ homogeneous component. In particular, the
vector subspaces
$$
\mathfrak{gl}^{j}(\gm ) := \{ A\in \mathfrak{gl}(\gm ),\  A(\gm^{i})\subset
\gm^{i+j},\ \forall i\}
$$
define a gradation of  $\mathfrak{gl}(\gm )$. This is a Lie algebra gradation: $[ \mathfrak{gl}^{j}(\gm ) , \mathfrak{gl}^{r}(\gm )]\subset
\mathfrak{gl}^{j+r} (\gm )$, for any $j, r.$ Consider the subalgebra
$\mathfrak{gl}_{m}(\gm):=\sum_{i\geq m} \mathfrak{gl}^{i} (\gm )$
and
$$
\mathrm{GL}_m(\gm ):= \{ B\in \mathrm{GL}(\gm ):\ B = \mathrm{Id} +
A,\  A\in \mathfrak{gl}_m(\gm ) \}$$ the Lie group with Lie algebra
$\mathfrak{gl}_{m} (\gm ).$  For $m\geq 2$,
 $GL_{m}(\gm )$ is a normal subgroup of $GL_{1} (\gm )$.
Any class $[A]\in GL_{1}(\gm ) / GL_{m}(\gm )$ 
is determined  by
the homogeneous components of $A$ up to degree $m-1.$

\begin{thm}\label{lifts-adaugat} i) The group $GL_{1}(\gm )$ acts simply transitively on $\mathrm{Gr}(V)$, by
$FA:= F\circ A$,  for any $F\in \mathrm{Gr}(V)$ and $A\in GL_{1}(\gm )$, and the orbits of the subgroup 
$GL_{m}(\gm)$ are  the fibers of the natural map
$\Pi^{m} :   \mathrm{Gr}(V) \rightarrow \mathrm{Gr}_{m} (V)$ defined by
(\ref{pi-m}).

\medskip

ii) The map $\Pi^{m}$ induces an isomorphism 
between the orbit space $ \mathrm{Gr} (V) / \mathrm{GL}_{m}(\gm )$ and $\mathrm{Gr}_{m}(V)$. 

\medskip 

iii)  The quotient group
$\mathrm{GL}_{1} (\gm ) / GL_{m}(\gm )$ acts 
simply transitively on $\mathrm{Gr}_{m} (V)$, by
\begin{equation}\label{f-bar-form}
 (F_{\bar{H}}[A])^{i} (x) :=
\sum_{j=0}^{m-1} f_{j+i,m} F^{j+i}_{\bar{H}} ( A^{j}(x)),\ \forall  x\in
\gm^{i},
\end{equation}
where $F_{\bar{H}} \in \mathrm{Gr}_{m}(V)$, $[A]\in
\mathrm{GL}_{1} (\gm ) / GL_{m}(\gm )$  and $f_{j+i}: \mathrm{gr}^{j+i}_{(m)}(V) \ra \mathrm{gr}^{i}_{(m)}(V)$
are  the natural maps.
\end{thm}

\begin{proof} 
Claim i) is easy, claim ii) follows from claim i) and the surjectivity of $\Pi^{m}$. We now prove iii). We define an action
of $GL_{1} (\gm ) / GL_{m}(\gm )$ on $\mathrm{Gr}_{m}(V)$ by $\Pi^{m} (F_{H}) [A] := \Pi^{m}( F_{H}\circ A)$,
for any $F_{H}\in \mathrm{Gr}(V)$ and $[A]\in GL_{1} (\gm ) / GL_{m}(\gm )$.
It is easy to check that it is a well-defined, simply transitive action. We now prove that it is given by (\ref{f-bar-form}).
To simplify notation, let $F_{\bar{H}} := \Pi^{m} (F_{H}).$ 
For any $x\in \gm^{i}$,
\begin{equation}\label{clar1}
(F_{\bar{H}} [A])(x)=\Pi^{m}(F_{H}A)(x)= \mathrm{pr}^{i}_{(m)}(
F_{H}\circ A)(x) =\sum_{j=0}^{m-1} (\mathrm{pr}^{i}_{(m)}\circ
F_{H}^{j+i}) ( A^{j}(x)).
\end{equation}
Consider  the left hand side of (\ref{f-bar-form}): for any fixed $0\leq j\leq m-1$,
\begin{equation}\label{clar2}
f_{j+i, m} F^{j+i}_{\bar{H}} (A^{j}(x))
= f_{j+i, m} \circ \mathrm{pr}^{j+i}_{(m)} \circ F_{H}^{j+i} (A^{j}(x))=
(\mathrm{pr}^{i}_{(m)} \circ F^{j+i}_{H})(A^{j}(x)),
\end{equation}
where we used 
(\ref{map-pi}) and $f_{j+i,m}\circ\mathrm{pr}^{j+i}_{(m)} =
\mathrm{pr}^{i}_{(m)}\vert_{V_{j+i}}$. Relation
(\ref{f-bar-form}) follows from (\ref{clar1}) and (\ref{clar2}). 
\end{proof}

\subsection{Tanaka prolongation of a non-positively graded Lie
algebra}\label{prol-sect-lie}

Let $\gm_{0} = \sum_{i=-k}^{-1}\gm^{i} +\gg^{0}$ be a
non-positively graded Lie algebra, with Lie bracket $[\cdot ,
\cdot ]$. We always assume that the negative part $\gm
:= \sum_{i=-k}^{-1}\gm^{i}$
of $\gm_{0}$  is fundamental, i.e. generated by
$\gm^{-1}$. We define inductively a sequence of vector
spaces $\gg^{r}$ ($r\geq 1$), such that,
with the notation $\gm_{f}:= \gm +
\sum_{r=0}^{f}\gg^{r}$ $(f\geq 0)$,
$\gg^{r} \subset \mathfrak{gl}^{r} (\gm_{r-1})$.
First, let
$$
\gg^{1}: = \{ A\in \mathfrak{gl}^{1}(\gm_{0}),\quad A[x,y]= [A(x),
y] + [ x, A(y)],\ \forall x, y\in \gm \} .
$$
Next, suppose that $\gg^{s}\subset\mathfrak{gl}^{s} (\gm_{s-1})$
are known for any $1\leq s\leq r$.
We define
\begin{equation}\label{g-l}
\gg^{r+1}:= \{ A\in \mathfrak{gl}^{r+1} (\gm_{r}),\ A[x,y] =
[A(x), y] + [x, A(y)] \quad \forall x,y\in \gm\}.
\end{equation}
In  (\ref{g-l})  $[\cdot , \cdot ]: \gm \times \gm_{r} \ra
\gm_{r}$ extends the Lie bracket $[\cdot , \cdot ]$ of $\gm$ and
\begin{equation}\label{mixed}
[x, z] = - [z, x] = - z(x),\quad  x\in \gm ,\ z\in \gg^{s}\subset
\mathfrak{gl}^{s}(\gm_{s-1}),\ s\leq r.
\end{equation}

Remark that any $A\in \gg^{r}\subset \mathfrak{gl}^{r}(\gm_{r-1})$
annihilates the non-negative part $\sum_{i=0}^{r-1}\gg^{i}$ of
$\gm_{r-1}$ and we may consider $\gg^{r}
\subset\mathrm{Hom}^{r} (\gm , \gm_{r-1}).$

\begin{thm}\cite{tanaka} The vector space $(\gm_{0})^{\infty}:= \gm_{0} + \sum_{r\geq 1}\gg^{r}$ has the structure
of a graded Lie algebra (called the {\bf Tanaka prolongation of $\gm_{0}$}),
with the following Lie bracket:

\medskip

i) the Lie bracket of two elements from $\gm_{0}$ is their Lie bracket in the Lie algebra $\gm_{0}$;

\medskip

ii) the Lie bracket $[x, z]$,  where  $x\in\gm$ and  $z\in
\gg^{s}$ ($s\geq 1$)   is given by (\ref{mixed}).

\medskip

iii) the Lie bracket 
$[f_{1}, f_{2}]$, where  $f_{1} \in 
\sum_{r\geq 0}\gg^{r}$ and $f_{2} \in  \sum_{r\geq 1}\gg^{r}$
is defined by induction by the condition  
$$
[f_{1}, f_{2}](x) = [ f_{1}(x), f_{2}] + [ f_{1}, f_{2}(x)],\quad
f_{1}\in \gg^{r_{1}},\ f_{2}\in \gg^{r_{2}},\ x\in \gm .
$$ 
\end{thm}

\begin{defn}
Let $G \subset GL (\gm )$ be a Lie group with Lie algebra
$\gg^{0}$. The group $G^{l}:= \mathrm{Id} +
\gg^{l}\subset \mathrm{End}(\gm_{l-1})$
with group operation $(\mathrm{Id}+A)(\mathrm{Id} + B) := \mathrm{Id} + A +B$
(for any $A, B\in \gg^{l}$) is called the {\bf
$l$-Tanaka prolongation of $G$}. 
\end{defn}

We denote by $G^{l} GL_{l+1}(\gm_{l-1})$ the subgroup of
$GL(\gm_{l-1})$ of all automorphisms of the form $\mathrm{Id} +
A^{l} + A_{l+1}$, where $A^{l}\in \gg^{l}\subset \mathfrak{gl}^{l}
(\gm_{l-1})$ and $A_{l+1}\in \mathfrak{gl}_{l+1} (\gm_{l-1}).$ 
The Tanaka prolongation $G^{l}$  is isomorphic to the quotient of  $G^{l}GL_{l+1} (\gm_{l-1})$
by the normal subgroup $GL_{l+1} (\gm_{l-1}).$

\subsection{$G$-structures}\label{torsion-G}

\begin{notation}\label{notatie}{\rm
We begin by fixing notation. Our actions on manifolds are always
right actions. If a Lie group $G$ acts on a manifold $P$, we
denote by $R_{g}  : P\ra P$, $p \ra pg$ the action of $g\in G$ on
$P$ and by $(\xi^{a})^{P}$ (or simply $\xi^{a}$) the fundamental
vector field on $P$ generated by $a\in \mathrm{Lie}(G)=\gg .$ For
any $u\in P$, $a,b\in \gg$ and $g\in G$, $(R_{g})_{*}
(\xi^{a}_{u}) = (\xi^{\mathrm{Ad}(g^{-1})(a)})_{ug}$
(see e.g. \cite{KN}, p. 51) and $[
\xi^{a}, \xi^{b} ] =
 \xi^{[a,b]}$ (see e.g. \cite{KN}, p. 41). 
In particular, if $\pi : P \ra M$ is a principal
$G$-bundle and $\nu :\gg \ra  T^{v}P$ the vertical parallelism,
$\nu (a)_{u} = \nu_{u}(a) := \xi^{a}_{u}$,
then $\nu_{ug}= (R_{g})_{*} \circ \nu_{u}\circ\mathrm{Ad}(g)$.}
\end{notation}

Let $\pi : P \rightarrow M$ be a $G$-structure with structure
group $G \subset GL(V)$. Any $u\in P$ is a frame $u : V \ra
T_{p}M$. The action of $g\in G$ on $u$ is given by $u g:=
u\circ g.$ Let $\theta \in \Omega^{1} (P, V)$ be the soldering
form of $\pi$, defined by
$\theta_{u}(X) : = (u^{-1}\circ \pi_{*})(X)$, for any  $X\in T_{u}P.$
It is well-known that $\theta$ is $G$-equivariant (see e.g.
\cite{sternberg}, p. 309-310):
\begin{equation}\label{inv-GL}
R_{g}^{*}(\theta ) = g^{-1}\circ \theta,\ L_{\xi^{A}} (\theta ) =
- A\circ \theta ,\quad g\in G,\ A\in \gg\subset \mathfrak{gl} (V).
\end{equation}

Let $\rho$ be a connection on the $G$-structure $\pi  :P \ra M$.

\begin{defn}\label{rho-twist} A {\bf $\rho$-twisted } vector field is a vector
field $X_{a}$ on $P$ (where $a\in V$), such that $(X_{a})_{u}\in T_{u}P$ is
the $\rho$-horisontal lift of $u(a)\in T_{\pi (u)}M$, for any $u\in P$.
\end{defn}

According to \cite{sternberg} (see p. 356),
\begin{equation}\label{behaviour}
(R_{g})_{*} X_{a}= X_{g^{-1} (a)},\ [ \xi^{B}, X_{a}]=
X_{B(a)},\quad g\in G,\ B\in \gg \subset\mathfrak{gl}(V),\ a\in V.
\end{equation}

\begin{defn} The {\bf $\rho$-torsion function} is the function
\begin{equation}\label{homog-prel}
t^{\rho} : P \ra \mathrm{Hom}(\Lambda^{2}(V),V),\ t^{\rho}_{u}
(a\wedge b) := (d\theta)_{u} (X_{a}, X_{b}),\quad u\in P,\  a\wedge
b\in \Lambda^{2} (V).
\end{equation}
\end{defn}

Remark that  $\theta
(X_{a})=a$ is constant, for any $a\in V$, and
\begin{equation}\label{homog-r1}
t^{\rho}_{u}(a\wedge b) = - \theta_{u} ([X_{a}, X_{b}]) = -(u^{-1}\circ
\pi_{*})( [ X_{a}, X_{b}]_{u}),\quad u\in P,\ a\wedge b\in
\Lambda^{2}(V).
\end{equation}

\begin{thm}\label{invariancy} i) The torsion function $t^{\rho}$ is
$G$-equivariant:
\begin{equation}\label{t-a-h}
t^{\rho}_{u g} (a\wedge b) = g^{-1} t^{\rho}_{u}(g (a) \wedge
g(b)),\quad u\in P,\ g\in G,\ a\wedge b\in \Lambda^{2}(V).
\end{equation}
ii) For any other connection $\rho^{\prime}$ on $\pi$,
\begin{equation}\label{t-rho-prime}
t^{\rho^{\prime}}_{u}(a\wedge b) =  t^{\rho}_{u}(a\wedge b) - A (
b) + B( a ),\ u\in P,\ a\wedge b\in \Lambda^{2}(V).
\end{equation}
Above $A, B\in \gg\subset \mathrm{End}(V)$ are given by $\xi^{A}_{u}:=
(X^{\prime}_{a})_{u}- (X_{a})_{u}$,  $\xi^{B}_{u}:=
(X^{\prime}_{b})_{u}- (X_{b})_{u}$, where
$X_{a}$, $X_{b}$ (respectively, $X_{a}^{\prime}$,
$X_{b}^{\prime}$)  are the
$\rho$-twisted  (respectively, the $\rho^{\prime}$-twisted) vector fields
determined by $a$, $b$.
\end{thm}

\subsection{Tanaka structures}\label{def-tanaka}

\subsubsection{Filtrations of the Lie algebra of vector fields}

Let $TM={\mathcal D}_{-k}\supset {\mathcal D}_{-k+1} \cdots
\supset {\mathcal D}_{l}$ ($l\geq -1$) be a flag of distributions
on a manifold $M$. For any $p\in M$, let
$\mathrm{gr}^{i}(T_{p}M)= {\mathcal D}^{i}_{p}:= (\mathcal
D_{i})_{p}/ ({\mathcal D}_{i+1})_{p}$,  $\mathrm{gr}(T_{p}M):=
\sum_{i} \mathrm{gr}^{i}(T_{p}M)$ and
$(\mathrm{gr}^{i})^{\mathcal D} : \mathcal
D_{i} \ra \mathrm{gr}^{i}(TM)$
the natural projection. We assume
that the non-positive part $\{ {\mathcal D}_{i},\ i\leq 0\}$
defines a filtration
$$
{\mathfrak X}(M)=\Gamma ({\mathcal D}_{-k}) \supset \Gamma
({\mathcal D}_{-k+1}) \supset \cdots \supset \Gamma ( {\mathcal
D}_{0})
$$
of the Lie algebra ${\mathfrak X}(M)$ of vector fields on $M$.
Then, for any $p\in M$, $\mathrm{gr}^{<0}(T_{p}M):= \sum_{i<  0}
\mathrm{gr}^{i}(T_{p}M)$ is a graded Lie algebra, with  Lie
bracket $\{ \cdot , \cdot\}_{p}$ (or just $\{\cdot , \cdot\}$ when $p$ is understood) 
induced by the Lie bracket of vector fields. It is called
the {\bf symbol algebra of $\{\mathcal D_{i}\}$} at $p$. The following  lemma
will be useful and can be checked directly.

\begin{lem}\label{usoara} Let $f: N \ra M$ be a smooth map of constant rank and
$\{ {\mathcal D}^{M}_{i},\ i \leq 0\}$  a flag of distributions
which defines a filtration of the Lie algebra ${\mathfrak X}(M)$.
Then $\{ {\mathcal D}^{N}_{i}= (f_{*})^{-1}({\mathcal D}^{M}_{i}),\
i\leq 0\}$ defines a filtration of the Lie algebra ${\mathfrak
X}(N)$. For any $X\in \Gamma ({\mathcal D}_{i}^{N})$, $Y\in \Gamma
({\mathcal D}_{j}^{N})$ with $i, j<0$ and $p\in N$,
$$
(\mathrm{gr}^{i+j})^{{\mathcal D}^M} f_{*} ( [X, Y]_{p}) = \{
(\mathrm{gr}^{i})^{{\mathcal D}^M} f_{*} (X_{p}),
(\mathrm{gr}^{j})^{{\mathcal D}^M} f_{*} (Y_{p})\}_{f(p)} .
$$
\end{lem}

\subsubsection{Definition of Tanaka structures}

Let $\gm_{0} = \sum_{i=-k}^{-1}\gm^{i} +\gg^{0}$ be a
non-positively graded Lie algebra, $(\gm_{0})^{\infty} = \gm_{0} +\sum_{i\geq 1}\gg^{i}$
its Tanaka prolongation and
 $(\gm_{l})^{\geq 0} = \sum_{i=0}^{l} \gg^{i}$ the non-negative
 part of $\gm_{l}= \gm_{0} +\sum_{i= 1}^{l}\gg^{i}.$

\begin{defn} A flag of distributions $TM={\mathcal D}_{-k}\supset
{\mathcal D}_{-k+1} \cdots \supset {\mathcal D}_{l}$ ($l\geq -1$)
is a {\bf filtration of type $\gm_{l}$} if the following conditions are satisfied:

\medskip

i) for any $i, j\leq 0$, $[ \Gamma (\mathcal D_{i}), \Gamma
(\mathcal D_{j} ) ] \subset \Gamma (\mathcal D_{i+j} )$;

\medskip

ii) for any $p\in M$, there is an isomorphism $u^{-}_{p}: \gm \ra
\mathrm{gr}^{<0}(T_{p}M)$ of graded Lie algebras;

\medskip

iii) for any $p\in M$, there is a {\bf canonical} isomorphism
$\nu_{p}: (\gm_{l})^{\geq 0}\ra \mathrm{gr}^{\geq 0}(T_{p}M)$ of
graded vector spaces.

\medskip

The isomorphism $u: = u^{-}_{p} \oplus \nu_{p}: \gm_{l} \ra
\mathrm{gr}(T_{p}M)$ is called a {\bf graded frame at $p$}.
\end{defn}

The group $\mathrm{Aut}(\gm )$ of automorphisms of the graded Lie
algebra $\gm$ acts simply transitively on the set $\mathbb{P}_{p}$
of graded frames at $p\in P$. We denote by $\pi : \mathbb{P} \ra
M$ the
principal bundle of graded frames. It has structure group
$\mathrm{Aut} (\gm )$.

\begin{defn}\label{def-tanaka-gen} Let $\{ {\mathcal D}_{i},\ -k\leq i\leq l\}$ be a filtration of type $\gm_{l}$ on a manifold $M$ and $G \subset
\mathrm{Aut}(\gm )$ a Lie subgroup of $\mathrm{Aut} (\gm )$.
A {\bf Tanaka $G$-structure of type $\gm_{l}$} on $M$ is a principal
$G$-subbundle $\pi_{G}: P_{G} \ra M$ of the bundle $\pi :
\mathbb{P}\ra M$ of graded frames.
\end{defn}

The notion of automorphism of a Tanaka structure is defined in a natural way:

\begin{defn}
An automorphism of a Tanaka $G$-structure $({\mathcal D}_{i},
\pi_{G}: P_{G}\ra M)$ of type $\gm_{l}$ is a diffeomorphism $f: M
\ra M$ with the following properties:

\medskip

i) it preserves the flag of distributions $\mathcal D_{i}$ (and
induces a map  $f_{*}: \mathrm{gr}(TM)\ra \mathrm{gr}(TM)$);

\medskip

ii) for any graded frame $u: \gm_{l} \ra \mathrm{gr}(T_{p}M)$ from
$P_{G}$, the composition $f_{*} \circ u: \gm_{l} \ra
\mathrm{gr}(T_{f(p)}M)$ is also a graded frame from $P_{G}.$
\end{defn}

Let $(\mathcal D_{i}, \pi_{G})$ be a Tanaka $G$-structure
of type $\gm = \sum_{i=-k}^{-1}\gm^{i}$  and  $\gg^{0} := \mathrm{Lie}(G).$
Since $\gg^{0} \subset\mathrm{Der}^{0} (\gm )$, $\gm (\gg^{0}):= \gm +
\gg^{0}$ is a graded Lie algebra: its Lie bracket $[ \cdot , \cdot
]$ extends the Lie brackets of $\gm$ and $\gg^{0}$ and $[a, b] =
- [b,a] = - b(a)$, for any $a\in \gm$ and $b\in \gg^{0}\subset\mathrm{End}(\gm ).$
Let $\gm (\gg^{0})^{\infty} := \gm (\gg^{0} )+ \sum_{l\geq 1}\gg^{l}$ be the
Tanaka prolongation of $\gm (\gg^{0}).$

\begin{defn} The Tanaka $G$-structure $(\mathcal D_{i}, \pi_{G})$ of type $\gm$ has
(finite) {\bf order} $\bar{l}$ if $\bar{l}$ is the minimal
number such that $\gg^{\bar{l}+1}=0$.
\end{defn}

\section{Statement of the main results}\label{statements}

In this paper we  aim to prove the following statements:

\begin{thm}\label{sf0} Let $ ({\mathcal D}_{i}, \pi_{G}: P_{G} \ra M)$ be a Tanaka $G$-structure of type
$\gm = \sum_{i = -k}^{-1}\gm^{i}$,  $\gm (\gg^{0})^{\infty} = \gm +
\sum_{i\geq 0}\gg^{i}$ the Tanaka prolongation  of $\gm (\gg^{0})= \gm + \gg^{0}$
(where $\gg^{0} = \mathrm{Lie}(G)$) and $G^{n}=
\mathrm{Id} +\gg^{n}$ the $n$-prolongation of $G$. There is a
sequence of principal $G^{n}$-bundles
$\bar{\pi}^{(n)} : \bar{P}^{(n)} \ra \bar{P}^{(n-1)}$ ($n\geq 1$),
with the following properties:

\medskip

${\bf A)}$ The base $\bar{P}^{(n-1)}$ has a Tanaka $\{
e\}$-structure of type $\gm_{n-1}$. This means that  there is a
flag of distributions $\{T\bar{P}^{(n-1)}= \bar{\mathcal
D}^{(n-1)}_{-k}\supset \cdots \supset \bar{\mathcal
D}^{(n-1)}_{n-1}\}$ which satisfies
$$
[\Gamma ( \bar{\mathcal D}^{(n-1)}_{i}), \Gamma (\bar{\mathcal
D}^{(n-1)}_{j})] \subset \Gamma ( \bar{\mathcal
D}^{(n-1)}_{i+j}),\ i,j\leq 0,
$$
and for  any $\bar{H}^{n-1}\in \bar{P}^{(n-1)}$, there is a
canonical graded vector space  isomorphism
$$
I_{\bar{H}^{n-1}} :\gm_{n-1}\ra
\mathrm{gr}(T_{\bar{H}^{n-1}}\bar{P}^{(n-1)})
$$ of
whose restriction to $\gm$ is a Lie algebra
isomorphism onto $\mathrm{gr}^{<0} (T_{\bar{H}^{n-1}}\bar{P}^{(n-1)})$.

\medskip

${\bf B)}$ The principal bundle $\bar{\pi}^{(n)}$ is the
quotient of a $G$-structure $\tilde{\pi}^{n}: \tilde{P}^{n} \ra
\bar{P}^{(n-1)}$, with structure group $G^{n} GL_{n+1}(\gm_{n-1})$,
by the normal subgroup $GL_{n+1} (\gm_{n-1}).$
The $G$-structure $\tilde{\pi}^{n}$ is a subbundle of the bundle
$\mathrm{Gr}(T\bar{P}^{(n-1)})\ra \bar{P}^{(n-1)}$ of
adapted gradations of $T\bar{P}^{(n-1)}$
(the latter being  a $G$-structure, whose frames are lifts of the graded frames 
$I_{\bar{H}^{n-1}}$, $\bar{H}^{n-1} \in \bar{P}^{(n-1)}$).
In particular, $\bar{\pi}^{(n)}:
\bar{P}^{(n)} = \tilde{P}^{n} /
GL_{n+1}(\gm_{n-1})\ra \bar{P}^{(n-1)}$ is canonically isomorphic to a subbundle of the bundle
$\mathrm{Gr}_{n+1} (T\bar{P}^{(n-1)})\ra T\bar{P}^{(n-1)}$ of
$(n+1)$-quasi-gradations of $T\bar{P}^{(n-1)}$.

\medskip

${\bf C)}$ The torsion function $t^{\tilde{\rho}}$ of one (equivalently, any) connection $\tilde{\rho}$ 
on $\tilde{\pi}^{n}$ satisfies $t^{\tilde{\rho}} _{H^{n}}(a\wedge b) \in (\gm_{n-1})_{i-1}$, for any 
$H^{n}\in\tilde{P}^{n}$ and $a\wedge b\in \gm^{-1}\wedge \gg^{i}$ ($0\leq i\leq n-1$),  
and
\begin{equation}\label{tilde-main}
(t^{\tilde{\rho}}_{H^{n}})^{0}(a\wedge b) = -  [a,b],\quad
H^{n}\in \tilde{P}^{n},\ a\wedge b \in \gm^{-1}\wedge
(\sum_{i=0}^{n-1}\gg^{i}).
\end{equation}
(In (\ref{tilde-main}) $[a,b]$ denotes the Lie bracket of $a$ and $b$ in the Lie algebra
$\gm (\gg_{0})^{\infty}$).
\end{thm}

We reobtain the final result of the Tanaka's prolongation
procedure:

\begin{thm}\label{sf} Let $({\mathcal D}_{i}, \pi_{G}: P_{G} \ra M)$ be a Tanaka
$G$-structure of type $\gm = \sum_{i=-k}^{-1}$ and order $\bar{l}.$ The
$\bar{l}$-Tanaka prolongation $\bar{P}^{(\bar{l})}$ has a
canonical $\{ e\}$-structure. The  automorphism group
$\mathrm{Aut}({\mathcal D}_{i}, \pi_{G})$ of $({\mathcal D}_{i},
\pi_{G})$ is isomorphic to  the  automorphism group of this $\{
e\}$-structure. It is  a finite dimensional Lie
group with $\mathrm{dim}\mathrm{Aut}({\mathcal D}_{i}, \pi_{G})\leq
\mathrm{dim}(M) + \sum_{i=0}^{\bar{l}}\mathrm{dim}( \gg^{i}).$
\end{thm}

The remaining part of the paper is devoted to the proofs of
Theorem \ref{sf0} and \ref{sf}.

\section{The first prolongation of a Tanaka
structure}\label{1st-prol-sect}

Let $({\mathcal D}_{i}, \pi_{G} : P_{G} \ra M)$ be a Tanaka
$G$-structure of type $\gm = \gm^{-k} +\cdots + \gm^{-1}.$ In this
section we define the first principal bundle
$\bar{\pi}^{(1)}:\bar{P}^{(1)} \ra P$ 
 from Theorem \ref{sf0}.

\subsection{The $G$-structure $\pi^{1} : P^{1} \ra
P$}\label{g-tan}

To simplify notation, we  denote by $P:= P_{G}$ the total space of
$\pi_{G}$. Let $\nu^{0} : \gg^{0} \ra T^{v}P$ be
the  vertical parallelism of $\pi_{G}$, where $\gg^{0} =
\mathrm{Lie}(G).$  For any $i\leq -1$, let $\mathcal D_{i}^{P}:=(
\pi_{G})_{*}^{-1} (\mathcal D_{i})$ and $\mathcal D_{0}:= T^{v}
P$ the tangent vertical bundle of $\pi_{G}$. The sequence
\begin{equation}\label{TPG}
TP= {\mathcal D}^{P}_{-k}\supset {\mathcal D}^{P}_{-k+1}\supset
\cdots \supset {\mathcal D}^{P}_{-1}\supset {\mathcal D}_{0}^{P}
\end{equation}
defines  a filtration of the Lie algebra ${\mathfrak X}(P)$ of
vector fields on $P$ and the differential $(\pi_{G})_{*}$
induces a symbol algebra isomorphism
\begin{equation}\label{symbol-alg}
(\pi_{G})_{*}: \mathrm{gr}^{<0}(T_{u}P) \ra
\mathrm{gr}(T_{p}M),\quad u\in P,\ p= \pi_{G}(u).
\end{equation}

The next proposition can be
checked directly.

\begin{prop}\label{hat} Any point $u\in P$ defines an isomorphism
$$\hat{u} =  (\pi_{G})_{*}^{-1} \circ u +   \nu^{0}_{u}:
\gm_{0} = \gm +\gg^{0}\ra \mathrm{gr} (T_{u}P)=\mathrm{gr}^{<0} (T_{u}P) + T^{v}_{u}P.
$$
The set of isomorphisms $\{ \hat{u},\ u\in P\}$ is a Tanaka $\{e
\}$-structure  of type $\gm_{0}$ on $P$. \end{prop}

From  Theorem \ref{g-f-vspace} (applied to gradations), 
any  gradation $H = \{ H^{i}\} $ of $T_{u}P$ adapted to the
filtration (\ref{TPG}) determines a frame
\begin{equation}\label{tan-e}
F_{H}: \gm_{0}\ra T_{u}P ,\quad
F_{H}:= \widehat{H}\circ \hat{u},
\end{equation}
which lifts the graded frame $\hat{u}:\gm_{0} \ra
\mathrm{gr}(T_{u}P)$  (for the definition of $\widehat{H}$, see the comments before Proposition
\ref{functii}).  For any $a\in
\gg^{0}$, $F_{H}(a)
= \widehat{H} ( (\xi^{a})^{P}_{u})
= (\xi^{a})^{P}_{u}$. From Theorem
\ref{lifts-adaugat} i) we obtain:

\begin{prop}\label{pi-1}The principal bundle $\pi^{1} : P^{1} \rightarrow P$ of
adapted gradations of $TP$ is a $G$-structure with structure group
$GL_{1} (\gm_{0}).$ It consists of all frames
of $T_{u}P$ (for any $u\in P$) which are lifts of the canonical
graded frame $\hat{u}: \gm_{0}\ra \mathrm{gr} (T_{u} P)$.
\end{prop}

\subsection{The action of $G$ on $P^{1}$}\label{act-sub-1}

In this subsection we construct an action of $G$ on $P^{1}$ which
lifts the  action of $G$ on the total space $P= P_{G}$ of the principal $G$-bundle $\pi_{G}.$
For any $g\in G$, $R_{g} : P \ra P$ preserves the filtration  (\ref{TPG})
and induces a map $(R_{g})_{*} : \mathrm{gr} (T_{u}P) \ra \mathrm{gr} (T_{ug}P)$, for any $u\in P.$
Let
\begin{equation}\label{def-rho-0}
\rho : G \ra \mathrm{Aut} (\gm_{0}),\ \rho (g) (a + b):= g(a)
+\mathrm{Ad}(g)(b),\quad g\in G,\ a\in \gm,\ b\in \gg^{0}.
\end{equation}

\begin{prop}\label{action-g0} i) For any $u\in P$ and $g\in G$,
the frames $\hat{u}$ and $\widehat{ug}$ from Proposition \ref{hat}
are related by
\begin{equation}\label{hat-u}
\widehat{ug} = (R_{g})_{*}\circ \hat{u} \circ \rho (g) :
 \gm_{0}\ra \mathrm{gr}
(T_{ug} P).
\end{equation}
ii) There is an action of $G$ on $P^{1}$, which associates to any
frame $F_{H}: \gm_{0}\ra T_{u}P$ from $P^{1}$ and $g\in G$ the
frame
\begin{equation}\label{act}
F_{H} g := (R_{g})_{*}\circ F_{H} \circ \rho (g) :
\gm_{0} \ra T_{ug} P .
\end{equation}
iii) For any $a\in \gg^{0}$, the fundamental vector field
$(\xi^{a})^{P^{1}}$ of the above action of $G$  on $P^{1}$, generated by $a$, is $\pi^{1}$-projectable and
$(\pi^{1})_{*}(\xi^{a})^{P^{1}}
=(\xi^{a})^{P}.$ \end{prop}

\begin{proof}
Claim i) follows from the definition of $\hat{u}$, $\widehat{ug}$,
and $\nu^{0}_{ug} = (R_{g})_{*} \circ \nu_{u}^{0}\circ
\mathrm{Ad}(g)$.  For claim ii), one
checks that  $F_{H}g\in P^{1}$, i.e. is a lift of $\widehat{ug}$ (direct
computation, which uses that $F_{H}$ is a lift of $\hat{u}$ and
that $\rho$ is gradation preserving). Claim iii) follows from
$R_{g}\circ \pi^{1} = \pi^{1} \circ R_{g}$ (where we use the same notation $R_{g}$ for the actions of $g \in G$ on
$P^{1}$ and $P$). 
\end{proof}

\begin{lem}\label{invG} The soldering form $\theta^{1} \in \Omega^{1}
(P^{1}, \gm_{0})$ of $\pi^{1}$ is $G$-equivariant:
\begin{equation}\label{thetaG}
(R_{g})^{*} \theta^{1}= \rho (g^{-1}) \circ \theta^{1}, \quad
L_{(\xi^{a})^{P^{1}}}(\theta^{1}) = - \rho_{*}(a)\circ
\theta^{1},\quad g\in G,\ a\in \gg^{0}.
\end{equation}
\end{lem}

\begin{proof} From the definition of $\theta^{1}$ and $R_{g}\circ \pi^{1} = \pi^{1} \circ R_{g}$,
we obtain, for any $X_{{H}}\in T_{{H}}P^{1}$,
\begin{align*}
&((R_{g})^{*} \theta^{1} )(X_{{H}}) = \theta^{1}
((R_{g})_{*}X_{{H}}) = (F_{H}g)^{-1} (\pi^{1}\circ
R_{g})_{*} (X_{{H}}) \\
&= (\rho (g^{-1})\circ (F_{H})^{-1} \circ (R_{g^{-1}})_{*}\circ (\pi^{1}\circ
R_{g})_{*}) (X_{H})\\
&= (\rho(g^{-1}) \circ (F_{H})^{-1} \circ (\pi^{1})_{*}) (X_{H}) = (\rho(g^{-1}) \circ \theta^{1}) (X_{H}).
\end{align*}
The second relation
(\ref{thetaG})  is the infinitesimal version of the first.
\end{proof}

\subsection{The torsion function $t^{\rho}$ of
$\pi^{1}$}\label{torsion-1}

Let $\rho$ be a connection on the $G$-structure $\pi^{1} : P^{1}
\ra P$. In this section we study the properties of the torsion
function $t^{\rho}$, in connection with the gradation of
$\gm_{0}.$
Let $\{ X_{a},\ a\in \gm_{0}\}$
be the family of $\rho$-twisted vector fields
on $P^{1}$ (recall Section \ref{torsion-G}).
For any $a\in \gm_{0}$, $(X_{a})_{H}\in T_{H}P^{1}$ is
the $\rho$-horisontal lift of $F_{H}(a)\in T_{p}P$ (where $\pi^{1} (H) =p$);
when $a\in \gg^{0}$,  $X_{a}\in {\mathfrak X}(P^{1})$ is the $\rho$-horisontal lift of
$(\xi^{a})^{P}\in {\mathfrak X}(P)$.

\begin{prop}\label{non-neg}  The function $t^{\rho} : P^{1} \ra
\mathrm{Hom} (\Lambda^{2} (\gm_{0}), \gm_{0})$ has only
components of non-negative homogeneous degree. \end{prop}

\begin{proof}
For any $i\leq 0$, let ${\mathcal D}_{i}^{P^{1}}:=
(\pi^{1})_{*}^{-1} ({\mathcal D}_{i}^{P})$.
Since for any $H\in P^{1}$, $F_{H}: \gm_{0} \ra T_{u}P $ preserves
filtrations, $X_{a}\in \Gamma ({\mathcal
D}_{i}^{P^{1}})$, for any $a\in (\gm_{0})^{i}$ ($i\leq 0$).
Similarly, $X_{b}\in\Gamma ( {\mathcal D}_{j}^{P^{1}})$ for any
$b\in (\gm_{0})^{j}$ ($j\leq 0$).
The sequence $\{
{\mathcal D}_{i}^{P^{1}},\ i\leq 0\}$ defines a filtration of the
Lie algebra ${\mathfrak X}(P^{1})$.
It follows that $[X_{a},
X_{b}]\in \Gamma ({\mathcal D}_{i+j}^{P^{1}})$ and
$t^{\rho}_{H}(a,b) = - (F_{H})^{-1}(\pi^{1})_{*} ( [X_{a},
X_{b}])_{H})$ belongs to $(\gm_{0})_{i+j}$.
\end{proof}

\begin{thm}\label{step-1} i) For any $a\wedge b\in \Lambda^{2}(\gg^{0})$
and $H\in P^{1}$, $t^{\rho}_{H}(a\wedge b) =- [a,b]$.\

ii) For any $a\wedge b\in \Lambda^{2} (\gm_{0})$ and $H\in P^{1}$,
$(t^{\rho}_{H})^{0} (a,b) = -[a,b]$.
\end{thm}

\begin{proof} Let $a, b\in \gg^{0}.$ Then $X_{a}$, $X_{b}$ are the $\rho$-horisontal lifts of
the fundamental vector fields $(\xi^{a})^{P}$ and $(\xi^{b})^{P}$
on $P$. Thus, $[X_{a}, X_{b}]$ is $\pi^{1}$-projectable and
$(\pi^{1})_{*}[X_{a}, X_{b}] = [ (\xi^{a})^{P}, (\xi^{b})^{P}]  =
(\xi^{[a,b]})^{P}$. We obtain
$$
t^{\rho}_{H} (a, b) = - (F_{H})^{-1}(\pi^{1})_{*} ( [ X_{a},
X_{b}]_{H}) = - (F_{H})^{-1}( \xi^{[a,b]})^{P} = - [a,b].
$$
Claim i) follows.

For claim ii),  we distinguish two cases:  I) $a, b\in \gm$;
II) $a\in \gg^{0}$, $b\in \gm$.

Let  $a\in \gm^{i}$ and $b\in \gm^{j}$ ($i, j<0$). Then $X_{a} \in
\Gamma ( {\mathcal D}_{i}^{P^{1}})$, $X_{b}\in \Gamma (
{\mathcal D}_{j}^{P^{1}})$ and $[X_{a}, X_{b}]\in \Gamma (\mathcal D_{i+j}^{P^{1}})$.
Being a lift of $\hat{u}: \gm_{0} \ra
\mathrm{gr} (T_{u}P)$, the frame $F_{H}:\gm_{0} \ra T_{u}P$ is
filtration preserving and satisfies
\begin{equation}\label{frame-0}
(\mathrm{gr}^{s})^{{\mathcal D}^P}\circ
F_{H}\vert_{(\gm_{0})_{s}}= \hat{u}\circ \pi_{(\gm_{0})^{s}}
\vert_{(\gm_{0})_{s}},\ \pi_{(\gm_{0})^{s}}\circ (F_{H})^{-1}
\vert_{\mathcal D^{P}_{s}} = \hat{u}^{-1} \circ
(\mathrm{gr}^{s})^{{\mathcal D}^{P}}.
\end{equation}
Using
$(\pi^{1})_{*} ( [ X_{a}, X_{b}]_{H}) \in ({\mathcal
D}_{i+j}^{P})_{u}$ and the second relation (\ref{frame-0}), we
obtain
\begin{equation}\label{ultim0}
(t^{\rho}_{H})^{0} (a\wedge b)= - \pi_{(\gm_{0})^{i+j}}
(F_{H})^{-1} (\pi^{1})_{*} ( [ X_{a}, X_{b}]_{H})
 = - \hat{u}^{-1} \circ
(\mathrm{gr}^{i+j})^{{\mathcal D}^P}(\pi^{1})_{*} ( [ X_{a},
X_{b}]_{H}).
\end{equation}
On the other hand, from  Lemma \ref{usoara}, $(\pi^{1})_{*} (X^{a}_{H}) = F_{H}(a)$,
$(\pi^{1})_{*} (X^{b}_{H}) = F_{H}(b)$ and the first relation
(\ref{frame-0}),  we obtain
\begin{equation}\label{ultim1}
(\mathrm{gr}^{i+j})^{{\mathcal D}^P}(\pi^{1})_{*} ( [
X_{a}, X_{b}]_{H})  = \{ ( (\mathrm{gr}^{i})^{{\mathcal
D}^P}\circ  F_{H})(a),
((\mathrm{gr}^{j})^{{\mathcal D}^P}\circ  F_{H})(b)\} = \{ \hat{u}(a), \hat{u} (b)\} .
\end{equation}
Using  that $\hat{u}: \gm \ra \mathrm{gr}^{<0} (T_{u}P) $ is a Lie algebra
isomorphism, we  deduce, from (\ref{ultim0}) and
(\ref{ultim1}),  that
$(t^{\rho}_{H} )^{0}(a\wedge b)  = - [a,b]$,  as needed.

It remains to consider $a\in \gg^{0}$ and $b\in \gm .$ For this we use
the action of $G$ on $P^{1}$, defined in Subsection
\ref{act-sub-1}. From Proposition \ref{action-g0},
$(\xi^{a})^{P^{1}}$ is $\pi^{1}$-projectable and
$(\pi^{1})_{*}
(\xi^{a})^{P^{1}} = (\xi^{a})^{P}$.
Since  $a\in \gg^{0}$,  $X_{a}$ is the $\rho$-horisontal lift of $(\xi^{a})^{P}$.  Therefore, the vector
field $Y:= X_{a} - (\xi^{a})^{P^{1}}$ is $\pi^{1}$-vertical. We
write
\begin{equation}\label{comp0}
t_{H}^{\rho} (a\wedge b) = -\theta^{1} ( [X_{a}, X_{b}]_{H}) = -
\theta^{1} ( [ (\xi^{a})^{P^{1}}, X_{b}]_{H}) - \theta^{1} ( [ Y,
X_{b}]_{H}).
\end{equation}
We need to compute the right hand side of
(\ref{comp0}). From Lemma \ref{invG} and $\theta^{1} (X_{b}) =b$,
\begin{equation}\label{comput1}
\theta^{1} ( [(\xi^{a})^{P^{1}}, X_{b}]_{H}) = -
(L_{(\xi^{a})^{P^{1}}} \theta^{1})_{H} (X_{b}) = \rho_{*}(a) (b) =
a(b).
\end{equation}
In order  to compute  $\theta^{1} ( [ Y, X_{b}]_{H})$, we write
$Y= \sum_{s} f_{s} (\xi^{A_{s}})^{P^{1}}$, where $f_{s}$ are
functions on $P^{1}$ and $\{ A_{s}\}$ is a basis of $\mathfrak{gl}_{1}
(\gm_{0}).$ Then
\begin{align}
\theta^{1} ([ Y, X_{b}]_{H})\nonumber&= - \sum_{s} \theta^{1}_{H}
( X_{b}(f_{s}) (\xi^{A_{s}})^{P^{1}} + f_{s} [X_{b},
(\xi^{A_{s}})^{P^{1}}]) \\
\nonumber&= - \sum_{s} f_{s}(H) \theta^{1}_{H} ( [
X_{b},(\xi^{A_{s}})^{P^{1}}])=- \sum_{s}
f_{s}(H)L_{(\xi^{A_{s}})^{P^{1}}}(\theta^{1}) (X_{b})\\
\label{comput-a}& =  \sum_{s} f_{s}(H)A_{s}(b),
\end{align}
where in the second  equality we used that $(\xi^{A_{s}})^{P^{1}}$
is $\pi^{1}$-vertical (hence annihilated by $\theta^{1}$),
in the third equality we used that $\theta^{1} (X_{b}) =b$ is constant
and in the last equality we used the second relation (\ref{inv-GL}).
From (\ref{comp0}), (\ref{comput1}) and
(\ref{comput-a}), we obtain
\begin{equation}\label{rho-v}
t^{\rho}_{H} (a\wedge b) =  - a(b) -  A(b),\quad a\in
\gg^{0}\subset\mathfrak{gl}(\gm ),\ b\in \gm,
\end{equation}
where  $A= \sum_{s} f_{s}(H) A_{s}\in \mathfrak{gl}_{1} (\gm_{0})$ is uniquely determined
by $(X_{a})_{H} - (\xi^{a})^{P^{1}}_{H} = (\xi^{A})^{P^{1}}_{H}.$
Assume now that $b\in \gm^{i}$. From (\ref{rho-v}),
$t^{\rho}_{H}(a\wedge b)\in (\gm_{0})_{i}$ and (by projecting
(\ref{rho-v}) onto $\gm^{i}$) $(t^{\rho}_{H})^{0} (a, b) = - a(b)=
- [a, b]$ as required.
\end{proof}

\subsection{Variation of the torsion $t^{\rho}$ of
$\pi^{1}$}\label{var-torsion}

Let $\rho$ be a connection on $\pi^{1} : P^{1} \ra P$.

\begin{prop}\label{var-tor-comp} i) The degree  zero homogeneous component of $t^{\rho}: P^{1} \ra
\mathrm{Hom} (\gg^{0}\wedge \gm , \gm_{0})$ is independent of
$\rho$.\

ii) The degree zero and one homogeneous components of   $t^{\rho}:
P^{1} \ra \mathrm{Hom} (\Lambda^{2} (\gm ), \gm )$ are independent
of $\rho$.
\end{prop}

\begin{proof} Let $\rho^{\prime}$ be another connection on $\pi^{1}.$
For any $a\in \gm_{0}$, the $\rho$ and $\rho^{\prime}$-twisted
vector fields $X_{a}$ and $X^{\prime}_{a}$, at a point $H\in
P^{1}$, are related by $(X_{a}^{\prime})_{H} = (X_{a})_{H}
+(\xi^{A})^{P^{1}}_{H}$, where $A\in \ggl_{1}(\gm_{0})$
(the Lie algebra of the structure group $GL_{1}(\gm_{0})$ of $\pi^{1}$).
Similarly, for any $b\in \gm_{0}$, $(X_{b}^{\prime})_{H} =
(X_{b})_{H} +(\xi^{B})^{P^{1}}_{H}$, where $B\in
\ggl_{1}(\gm_{0})$. From Theorem \ref{invariancy},
\begin{equation}\label{t-rho-prime}
t^{\rho^{\prime}}_{H}(a\wedge b) =  t^{\rho}_{H}(a\wedge b) - A (
b) + B( a ).
\end{equation}
Let $a\in \gg^{0}$ and $b\in \gm^{i}$ ($i < 0$). Then $B(a)=0$,
$\mathrm{deg} (A( b ) )\geq i+1$. We obtain that the $\gm^{i}$-component of $A(b) - B(a)$ vanishes.
Claim i) follows. Let $a\in \gm^{i}$
and $b\in \gm^{j}$ with $i, j < 0$. Then $\mathrm{deg} (A( b )
)\geq j+1 >  i+j+1$ and $\mathrm{deg} (B(a ) )\geq i+1 > i+j+1$. We obtain that
the $\gm^{i+j}$ and $\gm^{i+j+1}$-components of
$A(b) - B(a)$ vanish. Claim ii) follows.
\end{proof}

We denote by $\mathrm{Tor} (\gm_{0} ) := \mathrm{Hom}
(\Lambda^{2}(\gm ), \gm_{0})$ the {\bf space of torsions}. It is a
graded vector space, with gradation $\mathrm{Tor}^{m} (\gm_{0} )
=\sum_{i,j} \mathrm{Hom} (\gm^{i}\wedge \gm^{j},(\gm_{0})^{i+j+m}).$
For any $H\in P^{1}$, we denote by $(t^{\rho}_{H})^{m}$ the
projection of $t^{\rho}_{H}$ onto $\mathrm{Tor}^{m} (\gm_{0} ).$

\begin{defn}\label{def-torsion-tan} Let $\rho$ be a connection on the $G$-structure $\pi^{1}: P^{1} \ra P$
associated to the Tanaka structure $\pi_{G} : P_{G}\rightarrow M.$
The function
$$
t^{1}: P^{1}\rightarrow \mathrm{Tor}^{1} (\gm_{0} ),\quad P^{1}
\ni H \ra t^{1}_{H}:=(t^{\rho}_{H})^{1}\in
\mathrm{Tor}^{1}(\gm_{0})
$$
is called the {\bf torsion function of the Tanaka structure
$({\mathcal D}_{i}, \pi_{G})$}.  \end{defn}

\begin{prop} The torsion function  is independent of the choice of $\rho$.
It is given by:
\begin{equation}\label{1-deg}
t^{1}_{H}(a,b) = - \pi_{\gm^{i+j+1}}(F_{H})^{-1} (\pi^{1})_{*}  ( [X_{a},
X_{b}]_{H}),\quad H\in P^{1},\ a\in \gm^{i},\ b\in
\gm^{j},\ (i,j<0),
\end{equation}
where $X_{a}$, $X_{b}\in {\mathcal X}(P^{1})$ are $\rho$-twisted vector fields.

\end{prop}

\begin{proof}
The first claim follows from Proposition \ref{var-tor-comp} ii).
Relation (\ref{1-deg}) follows from (\ref{homog-r1}).
\end{proof}

\begin{prop}\label{prel-prol} For any $H\in P^{1}$ and $A= \mathrm{Id}
+A_{1}\in GL_{1} (\gm_{0})$,
$$
t^{1}_{HA}=  t^{1}_{H} + \partial A .
$$
Above, $\partial A\in \mathrm{Hom} (\Lambda^{2}\gm , \gm_{0})$ is
given  by
\begin{equation}
(\partial A )(a\wedge b):=  A^{1}_{1} ([a,b] ) - [A^{1}_{1}(a), b] -
[a, A^{1}_{1}(b)],\ a\wedge b\in \Lambda^{2} (\gm ).
\end{equation}
\end{prop}

\begin{proof}
The inverse  $A^{-1}$ of $A$  is of the form $A^{-1} =
\mathrm{Id} + \tilde{A}_{1}$, where $\tilde{A}_{1} \in
\mathfrak{gl}_{1} (\gm_{0})$ and  $\tilde{A}^{1}_{1} = -
A_{1}^{1}$. We choose a connection $\rho$ on
$\pi^{1}.$ From Theorem \ref{invariancy}, for any $a\wedge b\in
\Lambda^{2}(\gm )$,
\begin{align*}
&t^{\rho}_{HA} (a\wedge b) = A^{-1} t^{\rho}_{H}(A(a)\wedge
A (b))\\
&= t^{\rho}_{H} (a\wedge b) + t^{\rho}_{H} (a\wedge
{A}_{1}(b)) + t^{\rho}_{H} ({A}_{1}(a)\wedge b) +
t^{\rho}_{H}
({A}_{1}(a)\wedge {A}_{1}(b))\\
&+ \tilde{A}_{1} \left( t^{\rho}_{H} (a\wedge b) + t^{\rho}_{H} (a\wedge
{A}_{1}(b)) + t^{\rho}_{H} ({A}_{1}(a)\wedge b) +
t^{\rho}_{H} ({A}_{1}(a)\wedge {A}_{1}(b))\right) .
\end{align*}
Let $a\in \gm^{i}$ and $b\in \gm^{j}$ ($i, j<0$).
Projecting the above equality to $\gm^{i+j+1}$ and using that
$t^{\rho}$ has only components of non-negative homogeneous degree
(see Proposition \ref{non-neg}),  we obtain
$$
t^{1}_{HA} (a\wedge b) = t^{1}_{H}(a\wedge b) + (t^{\rho}_{H})^{0}
(a\wedge {A}_{1}^{1}(b)) + (t^{\rho}_{H})^{0}
({A}_{1}^{1}(a)\wedge b) + \tilde{A}_{1}^{1} t^{0}_{H}(a\wedge b).
$$
Using $\tilde{A}_{1}^{1} = - A_{1}^{1}$ and
$(t^{\rho}_{H})^{0}(a\wedge b) = - [a,b]$, for any $a,b\in \gm_{0}$
(see Theorem \ref{step-1}), we obtain our claim.
\end{proof}

\subsection{The first prolongation}\label{first-prol-sect}

Let $(\mathcal D_{i} , \pi : P_{G}\rightarrow M)$ be a Tanaka
$G$-structure of type $\gm = \sum_{i=-k}^{-1}\gm^{i}$
and $t^{1}: P^{1}
\rightarrow\mathrm{Tor}^{1} (\gm_{0})$ its torsion function
(see Definition \ref{def-torsion-tan}). Let
\begin{equation}\label{def-partial}
\partial : \ggl_{1} (\gm_{0}) \ra \mathrm{Tor}^{1} (\gm_{0}),\
(\partial A)(a\wedge b) =  A^{1}( [a,b] ) - [A^{1}(a), b] -  [a,
A^1 (b)],\ a\wedge b\in\Lambda^{2} ( \gm ).
\end{equation}
Fix a complement $W$ of $\partial  (\ggl_{1} (\gm_{0}))$ in
$\mathrm{Tor}^{1}(\gm_{0}).$

\begin{prop}\label{tilde-1} The bundle
$\tilde{\pi}^{1}:\tilde{P}^{1}:=(t^{1})^{-1}(W) \rightarrow P$ is
a $G$-structure with structure group $G^{1} GL_{2}(\gm_{0})$. 
The torsion function $t^{\tilde{\rho}}$ of any connection $\tilde{\rho}$ on 
$\tilde{\pi}^{1}$ satisfies $t^{\tilde{\rho}}_{H}(a\wedge b) \in \gm^{-1} + \gg^{0}$, for any
$H\in \tilde{P}^{1}$ and $a\wedge b \in \gm^{-1} \wedge \gg^{0}$,  and  
$$
(t^{\tilde{\rho}}_{H})^{0} (a\wedge b) = - [a,b],\ a\wedge b\in
\gm^{-1}\wedge \gg^{0}.
$$
\end{prop}

\begin{proof}
The first  claim follows from Proposition \ref{prel-prol} and
 $\mathrm{Ker}(\partial ) = \ggl_{2}
(\gm_{0}) + \gg^{1}.$ The second claim follows from Proposition
\ref{non-neg} and Theorem \ref{step-1} (extend $\tilde{\rho}$ to a connection  on $\pi^{1}$).
\end{proof}

Let $\bar{P}^{(1)}:= \tilde{P}^{1}/ GL_{2} (\gm_{0})$. The map
$\bar{\pi}^{(1)}: \bar{P}^{(1)}
\rightarrow P$
induced by $\tilde{\pi}^{1}$ is a principal bundle with  structure
group $G^{1}.$

\begin{defn}\label{first-prol-def} The principal $G^{1}$-bundle $\bar{\pi}^{(1)}: \bar{P}^{(1)}
\rightarrow P$ is called the  {\bf first prolongation of the
Tanaka structure $(\mathcal D_{i}, \pi_{G})$}. \end{defn}

The next proposition concludes the first induction step from the proof of Theorem \ref{sf0}.

\begin{prop}\label{caract-quasi}
The principal bundle $\bar{\pi}^{(1)}: \bar{P}^{(1)} \ra P$ satisfies
properties A), B) and C) from Theorem \ref{sf0}. In particular, it
is canonically isomorphic to a subbundle of the bundle
$\mathrm{Gr}_{2} (TP)\rightarrow P$ of $2$-quasi-gradations of
$TP.$
\end{prop}

\begin{proof}
From Proposition \ref{hat}, property A) is satisfied. Properties B) and C) follow from
the definition of $\bar{\pi}^{(1)}$ and 
Proposition \ref{tilde-1}.
The statement about quasi-gradations
follows from Theorem \ref{lifts-adaugat} ii).
\end{proof}

\section{The $G$-structure  $\pi^{n+1}:{P}^{n+1}\ra
\bar{P}^{(n)}$}\label{pi-n}

We now assume that the principal bundles $\bar{\pi}^{(i)}: \bar{P}^{(i)}
\ra\bar{P}^{(i-1)}$ from Theorem \ref{sf0} are given, for any
$i\leq n.$ Our goal is to construct the principal bundle
$\bar{\pi}^{(n+1)} : \bar{P}^{(n+1)} \ra \bar{P}^{(n)}$ from this
theorem. In particular,
$\bar{P}^{(n)}$ needs to have a Tanaka $\{ e\}$-structure of type
$\gm_{n}.$ This is induced from $\bar{P}^{(n-1)}$, as follows.

\begin{lem}\label{first-pin}
The manifold $\bar{P}^{(n)}$ has a Tanaka $\{ e\}$-structure of
type $\gm_{n}.$ The flag of distributions $\{\bar{\mathcal
D}^{(n)}_{i},\ -k\leq i\leq n\}$ of this Tanaka structure is  $\bar{\mathcal
D}^{(n)}_{i}:= (\bar{\pi}^{(n)})_{*}^{-1} (\bar{\mathcal
D}^{(n-1)}_{i})$ $(-k\leq i\leq n-1)$ and $\bar{\mathcal
D}^{(n)}_{n}:= T^{v} \bar{P}^{(n)}=\mathrm{Ker}  (\bar{\pi}^{(n)})_{*}$. For any $\bar{H}^{n} \in
\bar{P}^{(n)}$, the canonical graded frame
$$
I_{\bar{H}^{n}}: \gm_{n}=\gm_{n-1} + \gg^{n}\ra \mathrm{gr}
(T_{\bar{H}^{n}}\bar{P}^{(n)})= \sum_{-k\leq i\leq n-1}
\mathrm{gr}^{i}(T_{\bar{H}^{n}} \bar{P}^{(n)} )+
T^{v}_{\bar{H}^{n}} \bar{P}^{(n)}
$$
is given by
\begin{equation}\label{isom}
I_{\bar{H}^{n}}\vert_{\gm_{n-1}} :=
(\bar{\pi}^{(n)})_{*}^{-1}\circ   I_{\bar{H}^{n-1}}, \quad
I_{\bar{H}^{n}}\vert_{\gg^{n}} :=\nu^{n}_{\bar{H}^{n}},
\end{equation}
where
$$
(\bar{\pi}^{(n)})_{*} :  \sum_{-k\leq i\leq n-1} \mathrm{gr}^{i}(T_{\bar{H}^{n}}
\bar{P}^{(n)} )\ra \mathrm{gr}(T_{\bar{H}^{n-1}}\bar{P}^{(n-1)})
$$
is the isomorphism  induced by  the differential of
$\bar{\pi}^{(n)}$,
$\bar{H}^{n-1} = \bar{\pi}^{(n)} (\bar{H}^{n})$, 
and $\nu^{n}_{\bar{H}^{n}}: \gg^{n} \ra
T^{v}_{\bar{H}^{n}} \bar{P}^{(n)}$ is the vertical parallelism of $\bar{\pi}^{(n)}.$

\end{lem}

\begin{proof}
The only non-trivial fact to check is that $I_{\bar{H}^{n}}: \gm \ra \mathrm{gr}^{<0}(T_{\bar{H}^{n}}\bar{P}^{(n)})$ preserves Lie brackets. For this, we use that
both $ (\bar{\pi}^{(n)})_{*} :  \mathrm{gr}^{<0}(T_{\bar{H}^{n}}
\bar{P}^{(n)} )\ra \mathrm{gr}^{<0}(T_{\bar{H}^{n-1}}\bar{P}^{(n-1)})$ and
$I_{\bar{H}^{n-1}} : \gm\ra \mathrm{gr}^{<0} (T_{\bar{H}^{n-1}}\bar{P}^{(n-1)})$
have this property.
\end{proof}

In the next sections we shall consider various adapted gradations and quasi-gradations
of $T\bar{P}^{(n)}$ or $T\bar{P}^{(n-1)}.$ They are always considered with respect to
the filtrations of the Tanaka structures of these manifolds.

\subsection{Definition and basic properties of
$\pi^{n+1}$}\label{basic-pi-n}

An important role in the prolongation procedure plays
a $G$-structure $\pi^{n+1} : P^{n+1} \ra
\bar{P}^{(n)}$  which we are going to define in this subsection.
Let $\bar{H}^{n}\in
\bar{P}^{(n)}$ and $H^{n+1}=\{ (H^{n+1})^{i},\ -k\leq i\leq n\}$
an adapted gradation of $T_{\bar{H}^{n}}\bar{P}^{(n)}$. It
projects to an adapted gradation $(\bar{\pi}^{(n)})_{*} (H^{n+1})
:= \{ (\bar{\pi}^{(n)})_{*} (H^{n+1})^{i},\ -k \leq i\leq  n-1\}$
of $T_{\bar{H}^{n-1}} \bar{P}^{(n-1)}$ (remark that $(H^{n+1})^{n}
= T^{v}  \bar{P}^{(n)}$ projects trivially to $T_{\bar{H}^{n-1}}
\bar{P}^{(n-1)}$). The adapted gradations ${H}^{n+1}$ and
$(\bar{\pi}^{(n)})_{*} (H^{n+1}) $ define frames which lift the
canonical graded frames $I_{\bar{H}^{n}}$ and $I_{\bar{H}^{n-1}}$
respectively (see Theorem \ref{g-f-vspace}, applied to gradations and lifts):
\begin{align}
\nonumber F_{H^{n+1}}&= \widehat{H^{n+1}}\circ I_{\bar{H}^{n}}:  \gm_{n} \ra T_{\bar{H}^{n}}\bar{P}^{(n)}\\
\label{frame-n} F_{ (\bar{\pi}^{(n)})_{*} (H^{n+1})} &= \widehat{
(\bar{\pi}^{(n)})_{*} (H^{n+1})} \circ I_{\bar{H}^{n-1}} :
\gm_{n-1} \ra T_{\bar{H}^{n-1}}\bar{P}^{(n-1)}.
\end{align}
As usual,  $F^{i}_{H^{n+1}}:=
F_{H^{n+1}}\vert_{ (\gm_{n})^{i}}$ $(i\leq n$) and similarly
$F_{(\bar{\pi}^{(n)})_{*} (H^{n+1})}^{i}:=
F_{(\bar{\pi}^{(n)})_{*} (H^{n+1})}\vert_{(\gm_{n-1})^{i}}$
$(i\leq n-1$). Recall that $\bar{P}^{(n)}\subset \mathrm{Gr}_{n+1} (T\bar{P}^{(n-1)}).$

\begin{defn}\label{pi-n-t} The manifold $P^{n+1}$
is  the set of all adapted gradations $H^{n+1}$ of
$T_{\bar{H}^{n}}\bar{P}^{(n)}$ (for any $\bar{H}^{n}\in
\bar{P}^{(n)}$), whose projection $(\bar{\pi}^{(n)})_{*}
(H^{n+1})$ to $T_{\bar{H}^{n-1}}\bar{P}^{(n-1)}$ is compatible
with the quasi-gradation $\bar{H}^{n}\in \mathrm{Gr}_{n+1}
(T_{\bar{H}^{n-1}}\bar{P}^{(n-1)})$ (where $\bar{H}^{n-1}:=
\bar{\pi}^{(n)}(\bar{H}^{n})$).
The map $\pi^{n+1} : P^{n+1} \ra
\bar{P}^{(n)}$ is the natural projection.
\end{defn}

More precisely, we set
$$
P^{n+1} =\cup_{\bar{H}^{n} \in \bar{P}^{(n)}}
\{ H^{n+1} \in\mathrm{Gr} (
T_{\bar{H}^{n}}\bar{P}^{(n)}),\ \Pi^{n+1}
(\bar{\pi}^{(n)})_{*}({H}^{n+1}) =\bar{H}^{n}\}
$$
where $\Pi^{n+1}:\mathrm{Gr} (T_{\bar{H}^{n-1}}\bar{P}^{(n-1)})
\rightarrow \mathrm{Gr}_{n+1} (T_{\bar{H}^{n-1}}\bar{P}^{(n-1)})$
is the map (\ref{pi-m}). Using the
first relation (\ref{frame-n}), we identify any $H^{n+1} \in P^{n+1}$ with the associated frame
$F_{H^{n+1}}.$
The next lemma describes
${P}^{n+1}$ as a submanifold of the frame manifold of
$\bar{P}^{(n)}$.
In Lemma \ref{ajut-p2} ii) below the map $F_{\bar{H}^{n}}$ is the $(n+1)$-lift of $I_{\bar{H}^{n-1}}$
determined by $\bar{H}^{n}\in \mathrm{Gr}_{n+1} ( T_{\bar{H}^{n-1}}\bar{P}^{(n-1)})$
(according to  Theorem \ref{g-f-vspace}):
\begin{equation}\label{bar-lift}
F_{\bar{H}^{n}} = (F_{\bar{H}^{n}}^{i}),\quad 
F_{\bar{H}^{n}}^{i}=
(\widehat{\bar{H}^{n}})^{i}\circ I_{\bar{H}^{n-1}}
: (\gm_{n-1})^{i} \ra \mathrm{gr}^{i}_{(n+1)}(T_{\bar{H}^{n-1}} \bar{P}^{(n-1)}),\ -k\leq i\leq n-1.
\end{equation}

\begin{lem}\label{ajut-p2} i) Let $H^{n+1} = \{ (H^{n+1})^{i},\ -k\leq
i\leq n\}$ be an adapted gradation of $T_{\bar{H}^{n}}
\bar{P}^{(n)}$ and $(\bar{\pi}^{(n)})_{*} (H^{n+1})$ its
projection to $T_{\bar{H}^{n-1}}\bar{P}^{(n-1)}.$ The associated
frames $F_{(\bar{\pi}^{(n)})_{*} (H^{n+1})}$ and $F_{H^{n+1}}$
defined by (\ref{frame-n}) are related by

\begin{equation}\label{rel-fr}
F_{ (\bar{\pi}^{(n)})_{*} (H^{n+1})} =
(\bar{\pi}^{(n)})_{*}\circ F_{H^{n+1}}\vert_{\gm_{n-1}}.
\end{equation}

ii)  The fiber of $\pi^{n+1}$ over $\bar{H}^{n}\in \bar{P}^{(n)}$
consists of all $H^{n+1} \in \mathrm{Gr}
(T_{\bar{H}^{n}}\bar{P}^{(n)})$ whose associated frame
$F_{H^{n+1}}$ satisfies: for any $-k\leq i\leq n-1$ and
$x\in (\gm_{n-1})^{i}$,
\begin{equation}\label{cat}
\mathrm{pr}^{i}_{(n+1)}(\bar{\pi}^{(n)})_{*} F^{i}_{H^{n+1}}(x) = F^{i}_{\bar{H}^{n}}(x),
\end{equation}
where $\mathrm{pr}^{i}_{(n+1)}: (\bar{\mathcal
D}^{(n-1)}_{i})_{\bar{H}^{n-1}} \ra
\mathrm{gr}^{i}_{(n+1)} (T_{\bar{H}^{n-1}}\bar{P}^{(n-1)})$
is the natural projection.
In particular, $(\bar{\pi}^{(n)})_{*} F_{H^{n+1}}
= F_{\bar{H}^{n}}$
on $(\gm_{n-1})_{-1}.$
\end{lem}

\begin{proof}
From the definitions of
$\widehat{H^{n+1}}$ and $\widehat{(\bar{\pi}^{(n)})_{*}
(H^{n+1})}$,
\begin{equation}\label{aditional}
\widehat{ (\bar{\pi}^{(n)})_{*} (H^{n+1})} \circ
(\bar{\pi}^{(n)})_{*}\vert_{ \mathrm{gr}^{\leq
n-1}(T_{\bar{H}^{n}} \bar{P}^{(n)})} = (\bar{\pi}^{(n)})_{*} \circ
\widehat{H^{n+1}}\vert_{\mathrm{gr}^{\leq n-1}(T_{\bar{H}^{n}}
\bar{P}^{(n)})}.
\end{equation}
Relation (\ref{rel-fr}) follows from (\ref{frame-n}),
(\ref{aditional}) and
$I_{\bar{H}^{n}}\vert_{\gm_{n-1}} = ( \bar{\pi}^{(n)})_{*}^{-1}
\circ I_{\bar{H}^{n-1}}.$

For claim ii), let $H^{n+1}\in \mathrm{Gr}(T_{\bar{H}^{n}}
\bar{P}^{(n)})$. Then $H^{n+1} \in P^{n+1}$ if and only if
$(\bar{\pi}^{(n)})_{*}(H^{n+1})\in
\mathrm{Gr}(T_{\bar{H}^{n-1}}\bar{P}^{(n-1)})$ is compatible with
the quasi-gradation $\bar{H}^{n}\in
\mathrm{Gr}_{n+1}(T_{\bar{H}^{n-1}}\bar{P}^{(n-1)}).$ From
Proposition \ref{functii} ii), this  condition is equivalent to
\begin{equation}\label{rel-hat}
\mathrm{pr}^{i}_{(n+1)} \circ  \widehat{ (\bar{\pi}^{(n)})_{*}
(H^{n+1})}^{i} = (\widehat{\bar{H}^{n}})^{i},\ i\leq n-1.
\end{equation}
Composing   (\ref{rel-hat}) with
$I_{\bar{H}^{n-1}}$ and using the relations (\ref{frame-n}) and (\ref{bar-lift}),
we obtain that (\ref{rel-hat}) is equivalent to
$\mathrm{pr}^{i}_{(n+1)} \circ F_{ ( \bar{\pi}^{(n)})_{*}(H^{n+1})}^{i} = F^{i}_{\bar{H}^{n}}$, or,
from  (\ref{rel-fr}), to (\ref{cat}).
\end{proof}

Below any $A\in \mathrm{Hom} (\sum_{i=0}^{n-1}\gg^{i}, \gg^{n})$ acts on
$\gm_{n}$, by annihilating $\gm$ and $\gg^{n}.$

\begin{prop}\label{pi-2} The projection $\pi^{n+1} :
P^{n+1}\ra \bar{P}^{(n)}$ is a $G$-structure with structure group
$\bar{G}:= \mathrm{Id} + \mathfrak{gl}_{n+1} (\gm_{n})
+\mathrm{Hom} (\sum_{i=0}^{n-1}\gg^{i}, \gg^{n})$.
\end{prop}

\begin{proof}
Let $H^{n+1}$, $\tilde{H}^{n+1}\in (\pi^{n+1})^{-1} (\bar{H}^{n})$
be two adapted gradations of $T_{\bar{H}^{n}} \bar{P}^{(n)}$, whose projections to $T_{\bar{H}^{n-1}} \bar{P}^{(n-1)}$ are
compatible with the quasi-gradation $\bar{H}^{n} \in
\mathrm{Gr}_{n+1} (T_{\bar{H}^{n-1}} \bar{P}^{(n-1)})$.
From Lemma \ref{ajut-p2} ii), for any $x\in (\gm_{n-1})^{i}$, $i\leq n-1$,
\begin{equation}\label{dif-f1}
F^{i}_{H^{n+1}} (x)- F^{i}_{\tilde{H}^{n+1}}(x) \in
(\bar{\pi}^{(n)})_{*}^{-1} (\bar{\mathcal
D}^{(n-1)}_{i+n+1})_{\bar{H}^{n-1}}= (\bar{\mathcal
D}^{(n)}_{i+n+1})_{\bar{H}^{n}} + T^{v}_{\bar{H}^{n}}
\bar{P}^{(n)}.
\end{equation}
Note  that $T^{v}_{\bar{H}^{n}} \bar{P}^{(n)}\subset
(\bar{\mathcal D}^{(n)}_{i+n+1})_{\bar{H}^{n}}$ when $i\leq -1$
and $(\bar{\mathcal D}^{(n)}_{i+n+1})_{\bar{H}^{n}}=0$ when $i\geq
0$.  Also,
\begin{equation}\label{dif-f2}
F^{n}_{H^{n+1}} = F^{n}_{\tilde{H}^{n+1}}: (\gm_{n})^{n} = \gg^{n}\ra
(\bar{\mathcal D}^{(n)}_{n})_{\bar{H}^{n}} = T^{v}_{\bar{H}^{n}}
\bar{P}^{(n)}
\end{equation}
is the vertical parallelism  of
$\bar{\pi}^{(n)}.$
From  relations (\ref{dif-f1}) and (\ref{dif-f2}) we obtain
$\mathrm{Id} + A : = F_{H^{n+1}}^{-1} \circ F_{\tilde{H}^{n+1}}\in \bar{G}$.
\end{proof}

\subsection{An action of $G^{n} GL_{n+1}(\gm_{n-1})$ on
$P^{n+1}$}\label{action-g-n}

In this subsection we define an action of $G^{n} GL_{n+1}(\gm_{n-1})$
on $P^{n+1}$, naturally related to the action of $G^{n}$ on the total space
$\bar{P}^{(n)}$ of the principal $G^{n}$-bundle $\bar{\pi}^{(n)}.$  Consider the group homomorphism
$$
\mathrm{Pr} : G^{n} GL_{n+1} (\gm_{n-1}) \ra G^{n},\quad  g =
\mathrm{Id}+A^{n} +A_{n+1}\ra \mathrm{Pr}(g):= \bar{g}=
\mathrm{Id}+ A^{n}.
$$
Let $\rho^{n}: G^{n} GL_{n+1} (\gm_{n})\ra \mathrm{Aut}(\gm_{n})$ be the trivial extension to  $\gm_{n} = \gm_{n-1} +\gg^{n}$ of the natural (left) action  of $G^{n} GL_{n+1}
(\gm_{n-1})\subset GL (\gm_{n-1})$ 
on $\gm_{n-1}$.  We define an action of  $G^{n}GL_{n+1}(\gm_{n-1})$  on the frame manifold
of $\bar{P}^{(n)}$: for any $g\in
G^{n}GL_{n+1} (\gm_{n-1})$ and frame $F:\gm_{n} \ra
T_{\bar{H}^{n}}\bar{P}^{(n)}$,
\begin{equation}\label{act-pn}
Fg :=  (R_{\bar{g}})_{*}\circ F \circ \rho^{n} (g): \gm_{n} \ra
T_{\bar{H}^{n}\bar{g}}\bar{P}^{(n)}.
\end{equation}

\begin{prop}\label{cheie}  The action (\ref{act-pn}) preserves $P^{n+1}$
and
\begin{equation}\label{cheie-r}
 ({\pi}^{n+1})_{*} ((\xi^{a})^{P^{n+1}}) = (\xi^{\bar{a}})^{\bar{P}^{(n)}},\quad\forall
a\in \gg^{n} + \mathfrak{gl}_{n+1} (\gm_{n-1}).
\end{equation}
(In (\ref{cheie-r}) 
$\bar{a}\in \gg^{n}$ denotes the 
$\gg^{n}$-component of $a$). 

\end{prop}

\begin{proof}
Let $H^{n+1} \in P^{n+1}$ and $F_{H^{n+1}} :\gm_{n} \ra
T_{\bar{H}^{n}}\bar{P}^{(n)} $ the associated frame.
We need to prove that for any $g\in G^{n}GL_{n+1} (\gm_{n-1})$, the frame
$F_{H^{n+1}}g$ related to $F_{H^{n+1}}$ as in
(\ref{act-pn}),
belongs to $P^{n+1}$, i.e.  satisfies the following conditions:\

I) it is a lift of $I_{\bar{H}^{n}\bar{g}}
:\gm_{n} \ra \mathrm{gr}(T_{\bar{H}^{n}\bar{g}}\bar{P}^{(n)})$,
i.e. is filtration preserving and
\begin{equation}\label{c-I}
((\mathrm{gr}^{i})^{\bar{\mathcal D}^{(n)}} \circ (F_{H^{n+1}}
g))(x) = (I_{\bar{H}^{n}\bar{g}} \circ \pi_{(\gm_{n})^{i}})(x),\
x\in (\gm_{n})_{i},\ i\leq n-1.
\end{equation}
(This means that $F_{H^{n+1}}g$ is the
frame associated to an adapted gradation of
$T_{\bar{H}^{n}\bar{g}}\bar{P}^{(n)}$).\

II) the adapted gradation from I)  
belongs to $P^{n+1}$,  i.e. (from Lemma \ref{ajut-p2}),
$$
\mathrm{pr}^{i}_{(n+1)}  (\bar{\pi}^{(n)})_{*} F^{i}_{H^{n+1}g}(x) = F^{i}_{\bar{H}^{n} \bar{g}}(x),\
\forall  x\in (\gm_{n-1})^{i},\ i\leq n-1.
$$

\medskip

Since $G^{n} GL_{n+1} (\gm_{n-1}) \subset GL_{1} (\gm_{n})$ and
$(R_{\bar{g}} )_{*}: T_{\bar{H}^{n-1}} \bar{P}^{(n)} \ra
T_{\bar{H}^{n-1}\bar{g}} \bar{P}^{(n)} $ preserve filtrations,
$F_{H^{n+1}}g$ preserves filtrations as well. Using the definition
of $F_{H^{n+1}} g$,
that $R_{\bar{g}}$ preserves filtrations,
$ (R_{\bar{g}^{-1}})_{*} \circ
I_{\bar{H}^{n}\bar{g}} = I_{\bar{H}^{n}}$ (which follows from
(\ref{isom}) and the fact that $\gg^{n}$ is abelian), we obtain that
(\ref{c-I})
is equivalent to
\begin{equation}\label{addd}
((\mathrm{gr}^{i})^{\bar{\mathcal D}^{(n)}} \circ
F_{H^{n+1}})(\rho^{n} (g)(x)) = (I_{\bar{H}^{n}} \circ
\pi_{(\gm_{n})^{i}})(x),\ x\in (\gm_{n})_{i},\ i\leq n.
\end{equation}
Using that 
$\rho^{n} (g) (x) \in (\gm_{n})_{i}$ and  $F_{H^{n+1}}$ lifts $I_{\bar{H}^{n}}$, 
we obtain that (\ref{addd}) is equivalent to  $\pi_{(\gm_{n})^{i}}
(\rho^{n} (g)(x) - x)=0$, which  holds  from the definition of
$\rho^{n}$. Condition I) is proved.

Condition II) can be checked in a similar way,  using
$$
F^{i}_{\bar{H}^{n}\bar{g}} (x)=
F^{i}_{\bar{H}^{n}} (x) + (f_{i+n,n+1} \circ
F^{i+n}_{\bar{H}^{n}})( A^{n}x),\ x\in (\gm_{n-1})^{i},\ i\leq
n-1
$$
where $A^{n}:= \bar{g} - \mathrm{Id}\in \gg^{n}$ and  
$f_{i+n,n+1}: \mathrm{gr}^{i+n}_{(n+1)} (T\bar{P}^{(n-1)}) \ra 
\mathrm{gr}^{i}_{(n+1)} ( T\bar{P}^{(n-1)})$
 is the natural map  (see
Theorem \ref{lifts-adaugat} ii)). We proved that (\ref{act-pn})
defines an action on $P^{n+1}.$ Relation (\ref{cheie-r}) follows
from $\pi^{n+1} \circ R_{g} = R_{\bar{g}} \circ \pi^{n+1}$, for
any $g\in G^{n}GL_{n+1} (\gm_{n-1}).$
\end{proof}

Let $\theta^{n+1} : TP^{n+1} \rightarrow \gm_{n}$ be the soldering
form of the $G$-structure $\pi^{n+1}$:
$$
\theta^{n+1} (X) =  (F_{H^{n+1}})^{-1} ((\pi^{n+1})_{*}X),\quad
\forall X\in T_{H^{n+1}} P^{n+1}.
$$
From relation  (\ref{inv-GL}),
it is $\bar{G}$-equivariant. The next
lemma shows that $\theta^{n+1}$ is equivariant also with respect
to the actions $\rho^{n}$   and (\ref{act-pn})
of $G^{n}GL_{n+1} (\gm_{n-1})$ on
$\gm_{n}$ and  $P^{n+1}$ respectively.

\begin{lem}\label{inv-sold} For any  $g\in G^{n} GL_{n+1}(\gm_{n-1})$
and $a\in \gg^{n} + \mathfrak{gl}_{n+1} (\gm_{n-1})$,
\begin{equation}\label{cheie-relations}
(R_g)^{*}(\theta^{n+1}) = \rho^{n} (g^{-1})\circ
\theta^{n+1},\quad L_{(\xi^{a})^{P^{n+1}}} (\theta^{n+1}) = -
(\rho^{n})_{*} (a) \circ \theta^{n+1}.
\end{equation}
\end{lem}

\begin{proof}
Like in the proof of Lemma \ref{invG}, for  any $g\in
G^{n}GL_{n+1} (\gm_{n-1})$,
$$
(R_g)^{*}(\theta^{n+1}) (X_{H^{n+1}}) = \theta^{n+1} ((R_g)_{*}
(X_{H^{n+1}}))= (  F_{H^{n+1}}g)^{-1}( (\pi^{n+1}\circ R_g)_{*}
(X_{H^{n+1}})).
$$
From $F_{H^{n+1}}g = (R_{\bar{g}})_{*}\circ F_{H^{n+1}} \circ
\rho^{n} (g)$ and $ \pi^{n+1} \circ R_g =R_{\bar{g}}\circ \pi^{n+1}$,
we obtain the first relation (\ref{cheie-relations}). The
second relation (\ref{cheie-relations}) is the infinitesimal
version of the first.

\end{proof}

\section{The torsion function of $\pi^{n+1}$}\label{torsion-g-n}

In this section we prove the following theorem.

\begin{thm}\label{main-t} Let $\rho$ be a connection on the $G$-structure $\pi^{n+1}$ and $t^{\rho}$ its torsion function.

i) Then 
$t^{\rho} : P^{n+1} \ra \mathrm{Hom}( (\gm^{-1}+ \gg^{n})\wedge
\gm_{n}, \gm_{n})$
has only homogeneous components of non-negative reduced degree,
i.e. for any $H^{n+1} \in P^{n+1}$ and $-k \leq i\leq n$,
$$
t^{\rho}_{H^{n+1}}( \gm^{-1} \wedge (\gm_{n})^{i}) \subset
(\gm_{n})_{i-1},\ t^{\rho}_{H^{n+1}}( \gg^{n} \wedge
(\gm_{n})^{i})\subset (\gm_{n})_{\mathrm{min}\{ n+i,n\}}.
$$
ii)  For any $H^{n+1}\in P^{n+1}$,
\begin{equation}\label{comp-0}
(t^{\rho}_{H^{n+1}})^{0}(a\wedge b) = - [a,b],\ \forall a\wedge b
\in \gm^{-1}\wedge \gm_{n}+ \gg^{n} \wedge \gm .
\end{equation}
\end{thm}

We divide the proof of the above theorem into three parts
(Subsections \ref{s1}, \ref{s2} and \ref{s3}), according to the
$\mathrm{Hom} (\gm^{-1} \wedge \gm , \gm_{n})$, $\mathrm{Hom}
(\gm^{-1} \wedge (\sum_{i=0}^{n-1} \gg^{i}), \gm_{n})$ and
$\mathrm{Hom}( \gg^{n}\wedge \gm_{n}, \gm_{n})$-valued components
of $t^{\rho}.$  Along the proof we shall use
the following notation: for any $a\in\gm_{n}$,
the $\rho$-twisted vector field on $P^{n+1}$  determined by $a$ will be denoted by
$X_{a}^{n+1}$; for any $a$, $b$ belonging to $\gm$ or $\gg^{i}$,
$[a,b]$  will always denote (as in the statement of Theorem \ref{main-t} above)
their Lie bracket in the Tanaka prolongation
$\gm (\gg^{0})^{\infty}$.

\subsection{The $\mathrm{Hom}(\gm^{-1} \wedge \gm ,\gm_{n})$-valued component}\label{s1}

\begin{prop}\label{part1} The torsion function
$t^{\rho} : P^{n+1} \ra \mathrm{Hom}( \gm^{-1}\wedge \gm ,
\gm_{n})$ has only homogeneous components of non-negative degree.
For any $H^{n+1}\in P^{n+1}$,
\begin{equation}\label{comp-0}
(t^{\rho}_{H^{n+1}})^{0}(a\wedge b) = - [a,b],\ \forall a\wedge b
\in \gm^{-1}\wedge \gm .
\end{equation}

\end{prop}

\begin{proof}
The argument is similar to  the proof of Proposition \ref{non-neg}
and Theorem \ref{step-1} ii).
The sequence $\mathcal D_{i}^{n+1} := (\pi^{n+1})_{*}^{-1} \bar{\mathcal D}^{(n)}_{i}$
($i\leq 0$) is a filtration of ${\mathfrak X}(P^{n+1}).$
For any $a\in \gm^{-1}$ and
$b\in \gm^{i}$,  $X_{a}^{n+1} \in \Gamma ({\mathcal
D}^{n+1}_{-1})$, $X_{b}^{n+1} \in \Gamma ({\mathcal
D}^{n+1}_{i})$
and $[X^{n+1}_{a}, X^{n+1}_{b}] \in \Gamma ( {\mathcal D}_{i-1}^{n+1})$.
Since  $F_{H^{n+1}}: \gm_{n} \ra T_{\bar{H}^{n}}
\bar{P}^{(n)}$ is filtration preserving, we obtain that
$t^{\rho}_{H^{n+1}}(a\wedge b) = - (F_{H^{n+1}})^{-1}
(\pi^{n+1})_{*} ([X_{a}^{n+1}, X_{b}^{n+1}])\in (\gm_{n})_{i-1}$, which
proves the first statement. We now prove
(\ref{comp-0}). Since
$F_{H^{n+1}}$ is a lift of $I_{\bar{H}^{n}}: \gm_{n} \ra
\mathrm{gr} (T_{\bar{H}^{n}} \bar{P}^{(n)})$,
for any $-k\leq s\leq n$, 
$$
((\mathrm{gr}^{s})^{\bar{\mathcal D}^{(n)}}\circ F_{H^{n+1}}
)\vert_{(\gm_{n})_{s}} = I_{\bar{H}^{n}} \circ
\pi_{(\gm_{n})^{s}}\vert_{(\gm_{n})_{s}},\  (\pi_{(\gm_{n})^{s}} \circ
(F_{H^{n+1}})^{-1} )\vert_{\bar{\mathcal D}^{(n)}_{s}} =
(I_{\bar{H}^{n}})^{-1} \circ (\mathrm{gr}^{s})^{\bar{\mathcal
D}^{(n)}}.
$$
From these relations  and $(\pi^{n+1})_{*} ( [ X_{a}^{n+1},
X_{b}^{n+1}]_{H^{n+1}}) \in (\bar{\mathcal
D}^{(n)}_{i-1})_{\bar{H}^{n}}$, we obtain:
\begin{align*}
(t^{\rho}_{H^{n+1}})^{0} (a\wedge b) &= - (\pi_{(\gm_{n})^{i-1}}\circ
(F_{H^{n+1}})^{-1}\circ (\pi^{n+1})_{*} )( [ X_{a}^{n+1},
X_{b}^{n+1}]_{H^{n+1}})\\
&= - ( (I_{\bar{H}^{n}})^{-1}\circ
(\mathrm{gr}^{i-1})^{\bar{\mathcal D}^{(n)}}\circ
(\pi^{n+1})_{*} )([ X^{n+1}_{a}, X^{n+1}_{b}]_{H^{n+1}})\\
& = - (I_{\bar{H}^{n}})^{-1}\{ ((\mathrm{gr}^{-1})^{\bar{\mathcal
D}^{(n)}}\circ F_{H^{n+1}})(a),
((\mathrm{gr}^{i})^{\bar{\mathcal D}^{(n)}}\circ  F_{H^{n+1}})(b)\}\\
&= - (I_{\bar{H}^{n}})^{-1} \{ I_{\bar{H}^{n}} (a),
I_{\bar{H}^{n}} (b)\} = - [ a,b]
\end{align*}
(we  used Lemma \ref{usoara} and that
$I_{\bar{H}^{n}}\vert_{\gm}: \gm \ra \mathrm{gr}^{<0}
(T_{\bar{H}^{n}}\bar{P}^{(n)})$ is a Lie algebra isomorphism).
\end{proof}

\subsection{The $\mathrm{Hom} (\gm^{-1} \wedge (\sum_{i=0}^{n-1}\gg^{i}) ,
\gm_{n})$-valued component of $t^{\rho}$}\label{s2}

Since $\bar{\pi}^{(n)}: \bar{P}^{(n)} \ra \bar{P}^{(n-1)}$
satisfies the conditions from Theorem
\ref{sf0}, it is the quotient of a $G$-structure $\tilde{\pi}^{n}: \tilde{P}^{n} \ra\bar{P}^{(n-1)}$
with structure group $G^{n}GL_{n+1} (\gm_{n-1})$, by the normal subgroup
$GL_{n+1} (\gm_{n-1}).$
In particular, $\bar{P}^{(n)} = \tilde{P}^{n}/GL_{n+1}(\gm_{n-1})$
and the fundamental vector field $(\xi^{c})^{\tilde{P}^{n}}\in {\mathfrak X}(\tilde{P}^{n})$ generated by
$c\in \gg^{n}+\mathfrak{gl}_{n+1} (\gm_{n-1})$ projects to the fundamental vector field
$(\xi^{\bar{c}})^{\bar{P}^{(n)}}\in {\mathfrak X}(\bar{P}^{(n)})$ generated by
$\bar{c}\in \gg^{n}$ (the $\gg^{n}$-component of $c$).
Let $\tilde{\rho}$ be a
connection on the $G$-structure $\tilde{\pi}^{n}$ and $X_{a}^{n}\in {\mathfrak
X}(\tilde{P}^{n})$ the $\tilde{\rho}$-twisted vector fields ($a\in
\gm_{n-1}$). From (\ref{behaviour}), for any $A\in
G^{n}GL_{n+1}(\gm_{n-1})$  and $c\in \gg^{n} + \mathfrak{gl}_{n+1}
(\gm_{n-1})$,
\begin{equation}\label{beh}
(R_A)_{*}(X_{a}^{n}) = X_{A^{-1}(a)}^{n},\quad [
(\xi^{c})^{\tilde{P}^{n}}, X_{a}^{n}] = X^{n}_{c(a)}.
\end{equation}
The first relation (\ref{beh}) implies that $X^{n}_{a}$ is
$GL_{n+1} (\gm_{n-1})$-invariant, for any $a\in
(\gm_{n-1})_{-1}$  (because $A\vert_{(\gm_{n-1})_{-1}}=\mathrm{Id}$, for any $A\in GL_{n+1}(\gm_{n-1})$) and
descends to a vector field $\widehat{X}_{a}^{n}$ on $\bar{P}^{(n)}.$ The following lemma
collects the main properties of the vector fields
$\widehat{X}_{a}^{n}$.

\begin{lem}\label{mai-jos} i)  For any $a\in \gm^{-1}$ and $b\in\gg^{i}$ (with  $0\leq
i\leq n-1$),
\begin{equation}\label{f2}
[ \widehat{X}^{n}_{a}, \widehat{X}^{n}_{b}]  =
\widehat{X}_{[a,b]}^{n}\ \mathrm{mod}(\bar{\mathcal D}^{(n)}_{i}).
\end{equation}
ii) For any $c\in \gg^{n}\subset\mathfrak{gl} (\gm_{n-1})$,  $a \in \gm^{-1}$ and
 $b\in \sum_{i=0}^{n-1}\gg^{i}$,
\begin{equation}\label{f1}
[ (\xi^{c})^{\bar{P}^{(n)}}, \widehat{X}^{n}_{a}]=
\widehat{X}^{n}_{c(a)},\ [ (\xi^{c})^{\bar{P}^{(n)}}, \widehat{X}^{n}_{b}]=0.
\end{equation}
iii) Let $H^{n+1} \in P^{n+1}$,  $\bar{H}^{n}= \pi^{n+1} (H^{n+1})
\in \bar{P}^{(n)}$ and $a\in (\gm_{n-1})_{-1}$. Then
\begin{equation}\label{relatii-cheie-0}
F_{H^{n+1}}(a) =  (\widehat{X}^{n}_{a})_{\bar{H}^{n}} \
\mathrm{mod} ( T^{v}_{\bar{H}^{n}} \bar{P}^{(n)}).
\end{equation}
\end{lem}

\begin{proof} i) Let $H^{n}\in \tilde{P}^{n}$,
$F_{H^{n}} : \gm_{n-1} \ra T_{\bar{H}^{n-1}} \bar{P}^{(n-1)}$ the associated frame
and $a\in \gm^{-1}$, $b\in \gg^{i}$
with $0\leq i\leq n-1$. From the property C) in Theorem \ref{sf0},
we know that
$t^{\tilde{\rho}}_{H^{n}}(a\wedge b) = - (F_{H^{n}})^{-1}
(\tilde{\pi}^{n})_{*} ( [X_{a}^{n}, X_{b}^{n}]_{H^{n}})$ belongs to
$(\gm_{n-1})_{i-1}$
and its projection onto $(\gm_{n-1})^{i-1}$ is equal to $-[a,b].$
Using $(\tilde{\pi}^{n})_{*} ((X^{n}_{[a,b]})_{H^{n}} )= F_{H^{n}} ( [a,b])$ we obtain
$$
(F_{H^{n}})^{-1}
(\tilde{\pi}^{n})_{*} \left( [X_{a}^{n}, X_{b}^{n}]_{H^{n}}-
(X^{n}_{[a,b]})_{H^{n}}\right)=(F_{H^{n}})^{-1} (\tilde{\pi}^{n})_{*} ( [X_{a}^{n},
X_{b}^{n}]_{H^{n}}) -  [a,b] \in (\gm_{n-1})_{i}.
$$
Thus, $(\tilde{\pi}^{n})_{*} ( [X_{a}^{n},
X_{b}^{n}]_{H^{n}} - (X^{n}_{[a,b]})_{H^{n}})$ belongs to
$F_{H^{n}} ( (\gm_{n-1})_{i})=(\bar{\mathcal
D}^{(n-1)}_{i})_{\bar{H}^{n-1}}$.  But
since $a, b, [a,b]\in (\gm_{n-1})_{-1}$, the vector fields $X^{n}_{a}$,
$X^{n}_{b}$ and $X^{n}_{[a,b]}$ project to $\bar{P}^{(n)}$ and
$$
(\tilde{\pi}^{n})_{*}
( [X_{a}^{n}, X_{b}^{n}]_{H^{n}} - (X^{n}_{[a,b]})_{H^{n}}) =
(\bar{\pi}^{(n)})_{*} ([\widehat{X}_{a}^{n},
\widehat{X}_{b}^{n}]_{\bar{H}^{n}} -
(\widehat{X}^{n}_{[a,b]})_{\bar{H}^{n}}).
$$
We deduce that
$$
(\bar{\pi}^{(n)})_{*} ([\widehat{X}_{a}^{n},
\widehat{X}_{b}^{n}]_{\bar{H}^{n-1}} -
(\widehat{X}^{n}_{[a,b]})_{\bar{H}^{n-1}})
\in (\bar{\mathcal D}^{(n-1)}_{i})_{\bar{H}^{n-1}},
$$
which implies (\ref{f2}), because $(\bar{\pi}^{(n)})^{-1}_{*} (
\bar{\mathcal D}^{(n-1)}_{i} )= \bar{\mathcal D}^{(n)}_{i}.$

ii) In order to prove (\ref{f1}), let  $c\in \gg^{n}\subset \mathfrak{gl}^{n}(\gm_{n-1})$ and $a\in
(\gm_{n-1})_{-1}.$ The vector fields $X^{n}_{a}$, $X^{n}_{c(a)}$
and $(\xi^{c})^{\tilde{P}^{n}}$ on $\tilde{P}^{n}$ project to the vector fields $\widehat{X}^{n}_{a}$,
$\widehat{X}^{n}_{c(a)}$ and  $(\xi^{c})^{\bar{P}^{(n)}}$
on $\bar{P}^{(n)}$  (and $c(a) =0$,
$X^{n}_{c(a)} =0$,  for any $c\in \gg^{n}$ and $a\in
(\gm_{n-1})_{0}$). Claim ii) follows by
projecting the second relation (\ref{beh}) on $\bar{P}^{(n)}$.

iii)   Let $H^{n+1} \in P^{n+1}$, $\bar{H}^{n} = \pi^{n+1} (H^{n+1})\in \bar{P}^{(n)}$ and choose
$H^{n} \in \tilde{P}^{n}$ which projects to  $\bar{H}^{n}.$
For any $a\in (\gm_{n-1})_{-1}$,
\begin{equation}\label{pr-2}
(\bar{\pi}^{(n)})_{*}( (\widehat{X}^{n}_{a})_{\bar{H}^{n}})=
(\tilde{\pi}^{n})_{*} ((X^{n}_{a})_{H^{n}}) =
F_{H^{n}}(a)  = F_{\bar{H}^{n}}(a),
\end{equation}
where in the last equality we used that  $H^{n}\in
(\tilde{\pi}^{n})^{-1} (\bar{H}^{n-1})\subset\mathrm{Gr}(
T_{\bar{H}^{n-1}}\bar{P}^{(n-1)})$ is compatible with $\bar{H}^{n}
\in\mathrm{Gr}_{n+1} (T_{\bar{H}^{n-1}}\bar{P}^{(n-1)})$
(in particular, $F_{H^{n}} = F_{\bar{H}^{n}}$ on $(\gm_{n-1})_{-1}$).
On the other hand,
since $H^{n+1}\in P^{n+1}$,
$(\bar{\pi}^{(n)})_{*} F_{H^{n+1}} (a) = F_{\bar{H}^{n}}(a)$ (from
Lemma  \ref{ajut-p2} and $a \in (\gm_{n-1})_{-1}$).  We obtain
$(\bar{\pi}^{(n)})_{*}( (\widehat{X}^{n}_{a})_{\bar{H}^{n}} )= (\bar{\pi}^{(n)})_{*} F_{H^{n+1}} (a)$,
which implies (\ref{relatii-cheie-0}).
\end{proof}

\begin{prop}\label{part2} The  function
$t^{\rho}: {P}^{n+1} \ra \mathrm{Hom} (\gm^{-1}\wedge
(\sum_{i=0}^{n-1}\gg^{i}), \gm_{n})$ has only  homogeneous
components of non-negative degree. For any $H^{n+1} \in P^{n+1}$,
\begin{equation}\label{preliminar}
(t^{\rho}_{H^{n+1}})^{0}(a\wedge b)=  - [a,b],\quad a\wedge b \in
\gm^{-1}\wedge (\sum_{i=0}^{n-1}\gg^{i}).
\end{equation}
\end{prop}

\begin{proof} Let $a ,b\in (\gm_{n-1})_{-1}$. From
relation (\ref{relatii-cheie-0}),
$(\pi^{n+1})_{*} ( (X_{a}^{n+1})_{H^{n+1}}) = F_{H^{n+1}} (a) = 
(\widehat{X}^{n}_{a})_{\bar{H}^{n}}$ and similarly
$(\pi^{n+1})_{*} ( (X_{b}^{n+1})_{H^{n+1}}) =
(\widehat{X}^{n}_{b})_{\bar{H}^{n}}$
modulo $T^{v}_{\bar{H}^{n}} \bar{P}^{(n)}$.
Therefore,
there are $A, B\in
\mathfrak{gl}_{n+1} (\gm_{n}) +\mathrm{Hom}
(\sum_{i=0}^{n-1}\gg^{i}, \gg^{n})$ (the Lie algebra of the
structure group $\bar{G}$ of $\pi^{n+1}$) and $c,d\in \gg^{n}$
(the Lie algebra of the structure group of $\bar{\pi}^{(n)}$),
such that
\begin{align}
\nonumber&(X_{a}^{n+1})_{H^{n+1}} =
(\widetilde{\widehat{X}^{n}_{a}})_{H^{n+1}} +
(\widetilde{(\xi^{c})^{\bar{P}^{(n)}}})_{H^{n+1}}
+(\xi^{A})^{P^{n+1}}_{H^{n+1}},\\
\label{x-a-b}&(X_{b}^{n+1})_{H^{n+1}} =
(\widetilde{\widehat{X}^{n}_{b}})_{H^{n+1}} +
(\widetilde{(\xi^{d})^{\bar{P}^{(n)}}})_{H^{n+1}} +
(\xi^{B})^{P^{n+1}}_{H^{n+1}}
\end{align}
(for a vector field $Z\in {\mathfrak X}(\bar{P}^{(n)})$,
we denote by $\tilde{Z}$ its $\rho$-horisontal lift
to $P^{n+1}$). Then
\begin{align*}
&t^{\rho}_{H^{n+1}} (a\wedge b)  = (d\theta^{n+1})_{H^{n+1}}
(\widetilde{\widehat{X}^{n}_{a}} +
\widetilde{(\xi^{c})^{\bar{P}^{(n)}}} + (\xi^{A})^{P^{n+1}} ,
\widetilde{\widehat{X}^{n}_{b}} +
\widetilde{(\xi^{d})^{\bar{P}^{(n)}}}+ (\xi^{B})^{P^{n+1}})\\
& = (\widetilde{\widehat{X}^{n}_{a}} +
\widetilde{(\xi^{c})^{\bar{P}^{(n)}}} +
(\xi^{A})^{P^{n+1}})_{H^{n+1}} (f)   -
(\widetilde{\widehat{X}^{n}_{b}}
+\widetilde{(\xi^{d})^{\bar{P}^{(n)}}}+
(\xi^{B})^{P^{n+1}})_{H^{n+1}}
(g )\\
&  - \theta^{n+1}([ \widetilde{\widehat{X}^{n}_{a}} +
\widetilde{(\xi^{c})^{\bar{P}^{(n)}}} + (\xi^{A})^{P^{n+1}} ,
\widetilde{\widehat{X}_{b}^{n}} +
\widetilde{(\xi^{d})^{\bar{P}^{(n)}}} +
(\xi^{B})^{P^{n+1}}]_{H^{n+1}}),
\end{align*}
where
\begin{align*}
&f (H^{n+1}):= \theta^{n+1}_{H^{n+1}}
(\widetilde{\widehat{X}^{n}_{b}}+
\widetilde{(\xi^{d})^{\bar{P}^{(n)}}} + (\xi^{B})^{P^{n+1}} )=
(F_{H^{n+1}})^{-1}({\widehat{X}^{n}_{b}}+
(\xi^{d})^{\bar{P}^{(n)}})\equiv b\\
& g(H^{n+1}):= \theta^{n+1}_{H^{n+1}}
(\widetilde{\widehat{X}^{n}_{a}}+
\widetilde{(\xi^{c})^{\bar{P}^{(n)}}} + (\xi^{A})^{P^{n+1}} )=
(F_{H^{n+1}})^{-1}(\widehat{X}^{n}_{a}+
(\xi^{c})^{\bar{P}^{(n)}})\equiv a
\end{align*}
and the sign '$\equiv$'  means modulo $\gg^{n}.$ (We used
(\ref{relatii-cheie-0}),
$(F_{H^{n+1}})^{-1}((\xi^{c})^{\bar{P}^{(n)}}) =c\in \gg^{n}$ and $(F_{H^{n+1}})^{-1}((\xi^{d})^{\bar{P}^{(n)}})=d\in
\gg^{n}$ ). We obtain
\begin{align}
\nonumber t^{\rho}_{H^{n+1}} (a\wedge b)  &\equiv - \theta^{n+1}([
\widetilde{\widehat{X}^{n}_{a}} +
\widetilde{(\xi^{c})^{\bar{P}^{(n)}}} + (\xi^{A})^{P^{n+1}} ,
\widetilde{\widehat{X}_{b}^{n}} +
\widetilde{(\xi^{d})^{\bar{P}^{(n)}}} +
(\xi^{B})^{P^{n+1}}]_{H^{n+1}})\\
\nonumber&\equiv -  (F_{H^{n+1}})^{-1} ( [ \widehat{X}^{n}_{a} +
(\xi^{c})^{\bar{P}^{(n)}},
\widehat{X}^{n}_{b} + (\xi^{d})^{\bar{P}^{(n)}}]_{\bar{H}^{n}})\\
\nonumber&\equiv -  (F_{H^{n+1}})^{-1} ( [ \widehat{X}^{n}_{a},
\widehat{X}^{n}_{b} ] _{\bar{H}^{n}}+ [ (\xi^{c})^{\bar{P}^{(n)}},
\widehat{X}^{n}_{b}] _{\bar{H}^{n}}+ [ \widehat{X}^{n}_{a},
(\xi^{d})^{\bar{P}^{(n)}}]_{\bar{H}^{n}})\\
\label{f0}&  \equiv-  (F_{H^{n+1}})^{-1} ( [ \widehat{X}^{n}_{a},
\widehat{X}^{n}_{b} ] _{\bar{H}^{n}}+ (\widehat{X}^{n}_{c(b)})_{\bar{H}^{n}} -
(\widehat{X}^{n}_{d(a)})_{\bar{H}^{n}}),
\end{align}
where $\bar{H}^{n} = \pi^{n+1} (H^{n+1})$ and
we used (\ref{f1}) (we remark that $c(b) =0$ when $b\in
\sum_{i=0}^{n-1}\gg^{i}$ and similarly for $d(a)$).
Suppose  now that $a\in \gm^{-1}$ and that $b\in\gg^{i}$ (with
$0\leq i\leq n-1$).  Using (\ref{f2}), (\ref{relatii-cheie-0}),
(\ref{f0}) and $c(b)=0$  we obtain
\begin{align*}
t^{\rho}_{H^{n+1}} (a\wedge b) & = - (F_{H^{n+1}})^{-1}
( \widehat{X}^{n}_{[a,b]} - \widehat{X}^{n}_{d(a)})\quad \mathrm{mod} (\gm_{n})_{i}\\
& =- [a, b] +  d(a)\quad \mathrm{mod} (\gm_{n})_{i}.
\end{align*}
Since $d\in \gg^{n} \subset \mathfrak{gl}^{n}( \gm_{n-1})$,
$d(a) \in \gg^{n-1}$. Also, $[a,b] = - b(a) \in \gg^{i-1}$.
We deduce that
$t^{\rho}_{H^{n+1}}\in \mathrm{Hom} (\gm^{-1} \wedge (\sum_{i
=0}^{n-1}\gg^{i}), \gm_{n})$ has only components of non-negative
homogeneous degree and  relation (\ref{preliminar}) holds, for any
$a\wedge b \in \gm^{-1} \wedge (\sum_{i=0}^{n-1} \gg^{i})$.
\end{proof}

\subsection{The $\mathrm{Hom}( \gg^{n}\wedge \gm_{n}, \gm_{n})$-valued component}\label{s3}

This is the last component of the torsion function $t^{\rho}$
which  needs to be studied, in order to conclude the proof of
Theorem \ref{main-t}.

\begin{prop}\label{part3} The  function $t^{\rho}: P^{n+1}\ra \mathrm{Hom}( \gg^{n}\wedge \gm_{n}, \gm_{n})$
has non-negative reduced homogeneous components and satisfies
$$
(t^{\rho}_{H^{n+1}})^{0} (a\wedge b)=-  [a,b],\quad \forall
a\wedge b \in \gg^{n}\wedge  \gm .
$$
\end{prop}

\begin{proof} Let $a\in \gg^{n}$ and $b\in \gm_{n}$.
Recall that $G^{n}GL_{n+1} (\gm_{n-1})$ acts on $P^{n+1}$
and the fundamental vector field  $(\xi^{a})^{P^{n+1}}$
of this action, generated by $a\in \gg^{n}
\subset\gg^{n} + \mathfrak{gl}_{n+1} (\gm_{n-1})$, is $\pi^{n+1}$-projectable and
$(\pi^{n+1})_{*}(\xi^{a})^{P^{n+1}} = (\xi^{a})^{\bar{P}^{(n)}}$
(see Proposition \ref{cheie}).  On the other hand, $X_{a}^{n+1}\in {\mathfrak X}(P^{n+1})$ is
the $\rho$-horisontal lift of $(\xi^{a})^{\bar{P}^{(n)}}$. We obtain that
$Y:= X_{a}^{n+1} - (\xi^{a})^{P^{n+1}}$ is $\pi^{n+1}$-vertical.
We write
\begin{align}
\nonumber t^{\rho}_{H^{n+1}}(a\wedge b) &= - \theta^{n+1} (
[X_{a}^{n+1}, X_{b}^{n+1}]_{H^{n+1}})\\
\label{c}& = - \theta^{n+1} ([
(\xi^{a})^{P^{n+1}}, X_{b}^{n+1}
]_{H^{n+1}}) - \theta^{n+1}( [ Y, X_{b}^{n+1}]_{H^{n+1}}).
\end{align}
We need to compute the last row from the right hand side of
(\ref{c}). For the first term, we use Lemma \ref{inv-sold}
and that $\theta^{n+1} (X_{b}^{n+1}) = b$ is constant:
\begin{equation}\label{c1}
\theta^{n+1} (
[(\xi^{a})^{P^{n+1}}, X_{b}^{n+1}])\\
= - L_{(\xi^{a})^{P^{n+1}}}(\theta^{n+1}) (X_{b}^{n+1}) =
(\rho^{n})_{*}(a)(b).
\end{equation}
To compute the second term, we remark that,
since $Y$ is $\pi^{n+1}$-vertical, there is $A\in
\mathfrak{gl}_{n+1}(\gm_{n})+\mathrm{Hom}
(\sum_{i=0}^{n-1}\gg^{i}, \gg^{n})$, such that $Y_{H^{n+1}} =
(\xi^{A})^{P^{n+1}}_{H^{n+1}}.$ The soldering form $\theta^{n+1}$
of $\pi^{n+1}$ is $\bar{G}$-equivariant (see relation
(\ref{inv-GL})). Like in the computation
(\ref{comput-a}) from the proof of Theorem \ref{step-1}, we obtain
$\theta^{n+1}( [ Y, X_{b}^{n+1}]_{H^{n+1}}) =  A(b).$
This fact, together with (\ref{c}) and (\ref{c1}),  imply that
$$
t^{\rho}_{H^{n+1}} (a\wedge b) =-
 (\rho^{n})_{*}(a) (b) - A(b),\quad a\in \gg^{n}, \ b\in
 \gm_{n}.
$$
If $b\in \gm^{i}$ (with $i\leq -1$) then  $(\rho^{n})_{*}(a) (b) =
a(b) \in (\gm_{n-1})^{i+n}$  and
 $A(b) \in (\gm_{n})_{i+n+1}$. We deduce that
 $t^{\rho}_{H^{n+1}}(a\wedge b)\in (\gm_{n})_{i+n}$ and
$(t^{\rho}_{H^{n+1}})^{0}(a\wedge b) = - a(b)= - [a,b]$. If $b\in
\gg^{j}$ ($0\leq j\leq n$) then $(\rho^{n})_{*}(a)(b)=0$ and
$t^{\rho}_{H^{n+1}}(a\wedge b) = - A(b)\in \gg^{n}$.
\end{proof}

The proof of Theorem \ref{main-t} is now completed.

\section{Variation of the torsion $t^{\rho}$ of
${\pi}^{n+1}$}\label{variation-g-n}

In this section we define and study the $(n+1)$-torsion of the Tanaka structure
$(\mathcal D_{i}, \pi_{G})$. We
preserve the setting from  Section \ref{torsion-g-n}.
In particular, $\rho$ is a connection on the $G$-structure $\pi^{n+1} : P^{n+1} \ra \bar{P}^{(n)}.$

\begin{prop}\label{deg-2} i) Let $0\leq i\leq n-1$.
The map
$$
t^{\rho} : P^{n+1} \ra
\mathrm{Hom} (\gm^{-1}\wedge \gg^{i}, (\gm_{n})^{-i+1} +\cdots + (\gm_{n})^{n-1})
$$
is independent of the connection $\rho$.

\medskip

ii) Let $i\leq n+1$. The homogeneous component $(t^{\rho})^{i}$ of degree $i$
of $t^{\rho}: P^{n+1} \ra \mathrm{Hom}(\gm^{-1} \wedge \gm ,
\gm_{n})$ is independent of the connection $\rho$.

\end{prop}

\begin{proof} Consider  another connection $\rho^{\prime}$ on
$\pi^{n+1} : P^{n+1} \ra \bar{P}^{(n)}.$ From Theorem
\ref{invariancy} ii), for any $H^{n+1}\in P^{n+1}$ and $a,b\in
\gm_{n}$, there are $A, B\in\mathfrak{gl}_{n+1}(\gm_{n})+
\mathrm{Hom} (\sum_{i=0}^{n-1}\gg^{i}, \gg^{n})$, such that
\begin{equation}\label{rel-degree}
t^{\rho^{\prime}}_{H^{n+1}}(a\wedge b) =
t^{\rho}_{H^{n+1}}(a\wedge b) -  A(b) +  B(a).
\end{equation}
If $a\in \gm^{-1}$ and $b\in \gg^{i}$
($0\leq i\leq n-1$) then $A(b) , B(a)
\in \gg^{n}$ and, from (\ref{rel-degree}),   we obtain claim i). Let $a\in \gm^{-1}$ and $b\in
\gm^{j}.$ Then $\mathrm{deg} A(b) \geq n +1 + j > (-1 + j ) + i$
and $\mathrm{deg}B(a) = n > (-1+ j) +i$, for any $i\leq n+1$
(because $j<0$).
Relation (\ref{rel-degree}) again implies claim ii).
\end{proof}

\begin{defn}\label{torsion-n-def} i) The vector space
$\mathrm{Tor}^{n+1}(\gm_{n}):= \mathrm{Hom}^{n+1}
(\gm^{-1}\wedge \gm , \gm_{n})+ \sum_{i=0}^{n-1}\mathrm{Hom}
(\gm^{-1}\wedge \gg^{i}, \gg^{n-1}) $ is called the {\bf space of
$(n+1)$-torsions}.\

ii) Let $\rho$ be a connection on $\pi^{n+1}: P^{n+1}\ra
\bar{P}^{(n)}$. The function
$$
\bar{t}^{(n+1)}: P^{n+1} \rightarrow \mathrm{Tor}^{n+1}(\gm_{n})
$$
defined by
\begin{equation}\label{def-bar-t}
\bar{t}^{(n+1)}_{H^{n+1}}(a\wedge b) = \begin{cases}
(t^{\rho}_{H^{n+1}})^{n+1} (a\wedge b),\ a\wedge b\in \gm^{-1}\wedge \gm\\
(t^{\rho}_{H^{n+1}})^{n-i} (a\wedge b),\ a\wedge b\in \gm^{-1}\wedge \gg^{i},\\
\end{cases}
\end{equation}
for any $H^{n+1} \in P^{n+1}$ and $0\leq i\leq n-1$,
is called the {\bf $(n+1)$-torsion of the Tanaka structure  $(\mathcal D_{i}, \pi_{G})$}.
In (\ref{def-bar-t}) the expression $(t^{\rho}_{H^{n+1}})^{n-i} (a\wedge b)$,
for $a\wedge b\in \gm^{-1} \wedge \gg^{i}$, denotes
the projection of $t^{\rho}_{H^{n+1}} (a\wedge b)$ on $\gg^{n-1}.$
\end{defn}

From Proposition \ref{deg-2}, $\bar{t}^{(n+1)}$ is independent of
the choice of $\rho .$

\begin{thm}\label{modificare-tors} For any $H^{n+1} \in P^{n+1}$ and $\mathrm{Id}+A \in
\bar{G}$,
$$
\bar{t}^{(n+1)}_{H^{n+1}(\mathrm{Id}+ A)} =
\bar{t}^{(n+1)}_{H^{n+1}} + \partial^{(n+1)} A.
$$
Above
\begin{equation}\label{partial-map}
\nonumber{\partial}^{(n+1)} :\ggl_{n+1} (\gm_{n}) +
\sum_{i=0}^{n-1}\mathrm{Hom} (\gg^{i}, \gg^{n})\ra
\mathrm{Tor}^{n+1} (\gm_{n})
\end{equation}
maps $\ggl_{n+1} (\gm_{n})$ into $\mathrm{Hom}^{n+1}
(\gm^{-1}\wedge \gm , \gm_{n})$ and $\mathrm{Hom} (\gg^{i},
\gg^{n})$ into $\mathrm{Hom} (\gm^{-1}\wedge \gg^{i}, \gg^{n-1})$
($0\leq i\leq n-1$) and is defined by
\begin{align*}
&(\partial^{(n+1)} A_{n+1})(a\wedge b) := A_{n+1}^{n+1} ([a,b] )-
[A_{n+1}^{n+1}(a), b]- [a, A_{n+1}^{n+1}(b)],\ a\wedge b\in \gm^{-1}\wedge \gm\\
&(\partial^{(n+1)}  A^{n-i})(a\wedge b) := - [a, A^{n-i}(b)],\ a\wedge b\in
\gm^{-1}\wedge  \gg^{i},
\end{align*}
for any $A_{n+1} \in \mathfrak{gl}_{n+1} (\gm_{n})$ and $A^{n-i}
\in \mathrm{Hom} (\gg^{i}, \gg^{n})$.
\end{thm}

\begin{proof}
From Theorem \ref{invariancy},
\begin{equation}\label{inv-G-bar}
t^{\rho}_{H^{n+1}(\mathrm{Id}+A)} (a\wedge b) = (\mathrm{Id}+B)
t^{\rho}_{H^{n+1}}((\mathrm{Id}+A)( a)\wedge (\mathrm{Id}+A)(b)),\
a,b\in \gm_{n},
\end{equation}
where
$\mathrm{Id} + B := (\mathrm{Id}+A)^{-1}$. If $A = A_{n+1} + \sum_{i=1}^{n} A^{i}$ and
$B = B_{n+1} + \sum_{i=1}^{n} B^{i}$, with $A_{n+1}, B_{n+1} \in \mathfrak{gl}_{n+1} (\gm_{n})$ and
$A^{i}, B^{i} \in \mathrm{Hom} (\gg^{n-i}, \gg^{n})$, then
$B^{i}=  - A^{i}$ ($1\leq i\leq n$)
and $B_{n+1}^{n+1} = - A^{n+1}_{n+1}$
(easy check).
We write (\ref{inv-G-bar})
in the equivalent form
\begin{align*}
&t^{\rho}_{H^{n+1}(\mathrm{Id}+A)} (a\wedge b) =
t^{\rho}_{H^{n+1}} (a\wedge b) + t^{\rho}_{H^{n+1}} (a\wedge
{A}(b)) + t^{\rho}_{H^{n+1}} ({A}(a)\wedge b)\\
& + t^{\rho}_{H^{n+1}}
({A}(a)\wedge {A}(b))\\
&+ B \left( t^{\rho}_{H^{n+1}} (a\wedge b) + t^{\rho}_{H^{n+1}}
(a\wedge {A}(b)) + t^{\rho}_{H^{n+1}} ({A}(a)\wedge b) +
t^{\rho}_{H^{n+1}} ({A}(a)\wedge {A}(b))\right) .
\end{align*}
Suppose now that $a\in \gm^{-1}$ and $b\in \gm^{i}$ ($i<0$). The
above equality becomes
\begin{align}
\nonumber&t^{\rho}_{H^{n+1}(\mathrm{Id}+A)} (a\wedge b) =
t^{\rho}_{H^{n+1}} (a\wedge b) + t^{\rho}_{H^{n+1}} (a\wedge
{A}_{n+1}(b)) + t^{\rho}_{H^{n+1}} ({A}_{n+1}(a)\wedge b)\\
\nonumber& + t^{\rho}_{H^{n+1}}
({A}_{n+1}(a)\wedge {A}_{n+1}(b))\\
\nonumber&+ B \left( t^{\rho}_{H^{n+1}} (a\wedge b) +
t^{\rho}_{H^{n+1}} (a\wedge {A}_{n+1}(b)) + t^{\rho}_{H^{n+1}}
({A}_{n+1}(a)\wedge b)\right)\\
 \label{id-A}&+B \left(t^{\rho}_{H^{n+1}} ({A}_{n+1}(a)\wedge {A}_{n+1}(b))\right)
.
\end{align}
Since $a\in \gm^{-1}$ and $A_{n+1}(a) \in \gg^{n}$, all arguments
of $t^{\rho}_{H^{n+1}}$, in the right hand side of (\ref{id-A}),
belong to $(\gm^{-1}+\gg^{n})\wedge \gm_{n}$. From Theorem
\ref{main-t},
$$
t^{\rho}_{H^{n+1}} (a\wedge b) \in (\gm_{n})_{i-1},\
t^{\rho}_{H^{n+1}} (a\wedge {A}_{n+1}(b))\in (\gm_{n})_{i+n},\
t^{\rho}_{H^{n+1}} ({A}_{n+1}(a)\wedge b)\in (\gm_{n})_{i+n}.
$$
Also, since $A_{n+1}(a) \in \gg^{n}$ and $A_{n+1}(b) \in
(\gm_{n})_{i+n+1}$,
$$
t^{\rho}_{H^{n+1}} (A_{n+1}(a)\wedge A_{n+1}( b)) \in
(\gm_{n})_{\mathrm{min}\{ n, 2n+i+1\} }.
$$
We project (\ref{id-A}) on $(\gm_{n})^{i+n}$.
The term $t^{\rho}_{H^{n+1}}
(A_{n+1}(a)\wedge A_{n+1}( b)) $ brings no contribution
(because  $i+n <\mathrm{min}\{ n, 2n+i+1\}$).
We obtain
\begin{align}
\nonumber (t^{\rho}_{H^{n+1}(\mathrm{Id}+A)} )^{n+1}(a\wedge b)& =
(t^{\rho}_{H^{n+1}})^{n+1} (a\wedge b) +
(t^{\rho}_{H^{n+1}})^{0} (a\wedge {A}_{n+1}(b))\\
\label{explicatie-ad}& + (t^{\rho}_{H^{n+1}})^{0}
({A}_{n+1}^{n+1}(a)\wedge b)+ \pi_{(\gm_{n})^{i+n}} B t^{\rho}_{H^{n+1}} (a\wedge
b).
\end{align}
From
$t^{\rho}_{H^{n+1}}(a\wedge b)\in (\gm_{n})_{-1+i}$, $B
(\sum_{j=0}^{n-1}\gg^{j}) \subset \gg^{n}$, and $B\vert_{\gm} =
B_{n+1}\vert_{\gm}$, we obtain
\begin{equation}\label{nu-stiu}
\pi_{(\gm_{n})^{i+n}} B  t^{\rho}_{H^{n+1}} (a\wedge b)  = B_{n+1}^{n+1}  (t^{\rho}_{H^{n+1}})^{0} (a\wedge b).
\end{equation}
Using $B^{n+1}_{n+1} = - A^{n+1}_{n+1}$, relations (\ref{comp-0}),
(\ref{explicatie-ad}) and (\ref{nu-stiu}),  we obtain, for any
$a\wedge b\in \gm^{-1}\wedge \gm$,
\begin{equation}\label{f2-1}
(t^{\rho}_{H^{n+1}(\mathrm{Id}+A)})^{n+1}(a\wedge b) =
(t^{\rho}_{H^{n+1}} )^{n+1}(a\wedge b) + (\partial
A^{n+1})(a\wedge b).
\end{equation}
In  a similar way, we prove that, for any $a\wedge b\in
\gm^{-1}\wedge \gg^{i}$,
\begin{equation}\label{f2-2}
(t^{\rho}_{H^{n+1}(\mathrm{Id}+A)})^{n-i} (a\wedge b) =
(t^{\rho}_{H^{n+1}})^{n-i}(a\wedge b) -  [ a, A^{n-i}(b)].
\end{equation}
Relations (\ref{f2-1}) and (\ref{f2-2}) imply our claim.
\end{proof}

\section{Definition of
$\bar{\pi}^{(n+1)}:\bar{P}^{(n+1)}\ra
\bar{P}^{(n)}$}\label{prol-sect}

Consider the map $\partial^{(n+1)}$ from Theorem
\ref{modificare-tors}  and let $W^{n+1}$ be a
complement of $\mathrm{Im}({\partial}^{(n+1)})$ in
$\mathrm{Tor}^{n+1} (\gm_{n})$.

\begin{prop}\label{tilde-n}  i) The natural projection
$\tilde{\pi}^{n+1}:\tilde{P}^{n+1}=(\bar{t}^{(n+1)})^{-1}(W^{n+1})
\subset P^{n+1} \rightarrow \bar{P}^{(n)}$ is a $G$-structure,
with structure group $G= G^{n+1} GL_{n+2} (\gm_{n})$.\

ii) Let $\tilde{\rho}$ be a connection on $\tilde{\pi}^{n+1}$.
For any $H^{n+1} \in \tilde{P}^{n+1}$ and $a\wedge b\in \gm^{-1}\wedge \gg^{i}$
($0\leq i\leq n$), $t^{\tilde{\rho}}_{H^{n+1}} (a\wedge b) \in (\gm_{n})_{i-1}$ and
$$
(t^{\tilde{\rho}}_{H^{n+1}})^{0}(a\wedge b)= -  [a,b],\quad
H^{n+1}\in \tilde{P}^{n+1},\ a\wedge b \in \gm^{-1}\wedge
(\sum_{i=0}^{n}\gg^{i}).
$$
\end{prop}

\begin{proof} Any $A^{n-i}\in
\mathrm{Hom}(\gg^{i}, \gg^{n})$
$(0\leq i\leq n-1$)
with $\partial^{(n+1)}
(A^{n-i})=0$, i.e.
$$
[a, A^{n-i}(b)]= -A^{n-i}(b)(a)=0,\quad \forall a\in \gm^{-1},\
b\in \gg^{i},
$$
vanishes identically (because $A^{n-i}(b) \in \gg^{n}\subset
\mathrm{Hom} (\gm , \gm_{n-1})$ satisfies $A^{n-i}(b)[x,y]  = [
A^{n-i}(b)(x), y] + [ x, A^{n-i}(b)(y)]$, for any $x,y\in \gm$, and
$\gm^{-1}$ generates $\gm$; so, 
if  $A^{n-i} (b)\vert_{\gm^{-1}} =0$, for any $b\in \gg^{n}$,  
then $A^{n-i} (b) =0$ and $A^{n-i}=0$). We proved that
${\partial}^{(n+1)}\vert_{\sum_{i=0}^{n-1}\mathrm{Hom}(\gg^{i},
\gg^{n})}$ is injective. Similarly, any $A_{n+1} \in
\mathfrak{gl}_{n+1}(\gm_{n})$ which satisfies $\partial^{(n+1)}
(A_{n+1})=0$, i.e.
$$
A^{n+1}_{n+1}( [ a, b] )=[ A_{n+1}^{n+1}(a), b] + [ a, A_{n+1}^{n+1}(b)],\quad
\forall a\in \gm^{-1},\ b\in \gm ,
$$
satisfies this relation for any $a, b\in \gm .$ It follows that
$\mathrm{Ker} \left( \partial^{(n+1)}\vert_{\mathfrak{gl}_{n+1}
(\gm_{n})}\right) = \gg^{n+1}+\mathfrak{gl}_{n+2} (\gm_{n}).$
Claim i) follows. Claim ii) follows from Theorem \ref{main-t} ii) (extend
$\tilde{\rho}$ to a connection on $\pi^{n+1}$).

\end{proof}

We can finally define the map
$\bar{\pi}^{(n+1)} : \bar{P}^{(n+1)} \ra \bar{P}^{(n)}$ we are looking for. Namely,
let $\bar{P}^{(n+1)}: = \tilde{P}^{n+1} / GL_{n+2} (\gm_{n})$ and
$\bar{\pi}^{(n+1)} : \bar{P}^{(n+1)} \ra \bar{P}^{(n)}$ the map induced by
$\tilde{\pi}^{n+1}.$

\begin{prop}\label{prol-n-1} The map $\bar{\pi}^{(n+1)} : \bar{P}^{(n+1)} \ra
\bar{P}^{(n)}$ satisfies properties A), B), and C) from Theorem
\ref{sf0} (with $n$ replaced by $n+1$).
\end{prop}

\begin{proof}
From Lemma \ref{first-pin},  property A) is satisfied.
Property B) is satisfied by construction
and property C) follows from Proposition \ref{tilde-n}. From Theorem
\ref{lifts-adaugat}, $\bar{\pi}^{(n+1)}$ is
canonically  isomorphic to a
subbundle of the bundle $\mathrm{Gr}_{n+2} (T\bar{P}^{(n)})$ of
$(n+2)$-quasi-gradations of $T\bar{P}^{(n)}$.
\end{proof}

\section{Proof of Theorem \ref{sf}}\label{prelungire-tan}

In this section we prove Theorem \ref{sf}. In Subsection \ref{can-frame} we construct the canonical frame required by Theorem \ref{sf}.
In Subsection \ref{isom-sect} we prove the statements about the automorphism groups.

\subsection{The canonical frame of $\bar{P}^{(\bar{l})}.$}\label{can-frame}

\begin{prop}\label{ad-prop} Let $(\mathcal D_{i}, \pi_{G})$ be a Tanaka $G$-structure of
type $\gm = \sum_{i=-k}^{-1} \gm^{i}$ and
finite order $\bar{l}$. Then
the Tanaka prolongation $\bar{P}^{(\bar{l})}$ has a canonical frame $F^{\mathrm{can}}.$
\end{prop}

\begin{proof} Since
$\gg^{\bar{l}+1}=0$, also $\gg^{s} =0$ for any $s\geq \bar{l}+1$
and $\bar{\pi}^{(s)}:\bar{P}^{(s)} \ra \bar{P}^{(s-1)}$ is a
diffeomorphism. Moreover, for such an
$s$, $\bar{\mathcal D}^{(s)}_{\bar{l}+1}=0$ (at any $\bar{H}^{s}
\in \bar{P}^{(s)}$, $(\bar{\mathcal
D}^{(s)}_{\bar{l}+1})_{\bar{H}^{s}}$ is isomorphic to
$(\gm_{s})_{\bar{l}+1} = \gg^{\bar{l}+1}+\cdots + \gg^{s}$, which
is  trivial). We obtain  that $\mathrm{Gr}_{m} (T\bar{P}^{(s)}) =
\mathrm{Gr} (T\bar{P}^{(s)})$ for any $s\geq \bar{l}+ 1$ and
$m\geq k +\bar{l} + 1$ (see our comments
after Definition \ref{def-quasi}).
For any $f, t$ with $f\geq t+1$ we denote by $\bar{\pi}^{(f,t+1)}:
\bar{P}^{(f)} \ra \bar{P}^{(t)}$ the composition $\bar{\pi}^{(t+1)} \circ \cdots \circ
\bar{\pi}^{(f)}$.

We need
to construct  a canonical isomorphism
$F^{\mathrm{can}}_{\bar{H}^{\bar{l}}} : \gm_{\bar{l}}\ra T_{\bar{H}^{\bar{l}}}\bar{P}^{(\bar{l})}$,
for any $\bar{H}^{\bar{l}}\in \bar{P}^{(\bar{l})}$.
Let $\bar{H}^{k+\bar{l} + 1}:= ( \bar{\pi}^{(k+\bar{l}+1, \bar{l}+1)})^{-1} ( \bar{H}^{\bar{l}})
\in \bar{P}^{(k +\bar{l}+1)}.$
By our construction of Tanaka prolongations,  $\bar{P}^{(k+\bar{l}+1)} \subset
\mathrm{Gr}_{k+\bar{l} +2} (T\bar{P}^{(k+\bar{l})})$. From the
above, $\mathrm{Gr}_{k+\bar{l} +2} (T\bar{P}^{(k+\bar{l})}) =
\mathrm{Gr}(T\bar{P}^{(k+\bar{l})})$ and we obtain that
$\bar{P}^{(k+\bar{l}+1)} \subset \mathrm{Gr}
(T\bar{P}^{(k+\bar{l})})$. In particular,
$\bar{H}^{k+\bar{l}+1}$ defines a
gradation of $T_{\bar{H}^{k+\bar{l}}}\bar{P}^{(k+\bar{l})}$
or a
frame $F_{\bar{H}^{k+\bar{l}+1}}= \widehat{\bar{H}^{k+ \bar{l}+1}}\circ I_{\bar{H}^{k+\bar{l}}}
:\gm_{k+\bar{l}}= \gm_{\bar{l}}\ra
T_{\bar{H}^{k+\bar{l}}}\bar{P}^{(k+\bar{l})}$,
where $\bar{H}^{k+\bar{l}} := \bar{\pi}^{(k+ \bar{l}+1)} ( \bar{H}^{k+\bar{l} +1})$ and
 $I_{\bar{H}^{k+\bar{l}}}:
\gm_{\bar{l}}\ra
\mathrm{gr}(T_{\bar{H}^{k+\bar{l}}}\bar{P}^{(k+\bar{l})})$
is the graded frame  from
the Tanaka $\{ e\}$-structure of $\bar{P}^{(k+\bar{l})}.$
We define
$F^{\mathrm{can}}_{\bar{H}^{\bar{l}}}:= (\bar{\pi}^{(k+\bar{l}, \bar{l}+1)})_{*}\circ  F_{\bar{H}^{k+\bar{l} +1}}
:\gm_{\bar{l}} \ra T_{\bar{H}^{\bar{l}}} \bar{P}^{(\bar{l})}$.
\end{proof}

\subsection{The automorphism group $\mathrm{Aut} ({\mathcal D}_{i}, \pi_{G})$}\label{isom-sect}

The proof of the remaining part of Theorem \ref{sf} is based on the behaviour of the automorphisms of a Tanaka structure,
under the prolongation procedure:

\begin{prop}\label{ad-prop-1} Let $({\mathcal D}_{i}, \pi_{G} : P=P_{G} \ra M)$ be a Tanaka  $G$-structure of type
$\gm .$ The group of automorphisms $\mathrm{Aut} ({\mathcal D}_{i}, \pi_{G})$ of  $({\mathcal D}_{i}, \pi_{G})$
is isomorphic to the group of automorphisms
of the Tanaka $\{e\}$-structure on $P= P_{G}$ (see Proposition
\ref{hat}) and to  the group of automorphisms $\mathrm{Aut}(\tilde{\pi}^{n})$
of the $G$-structures $\tilde{\pi}^{n} : \tilde{P}^{n} \ra \bar{P}^{(n-1)}$, $n\geq 1.$
\end{prop}

\begin{proof} The argument is similar to the one used in Theorem
3.2 of \cite{sternberg} (in the setting of prolongation of $G$-structures) and is based on the naturality of our construction.
One first shows that any  $f\in \mathrm{Aut} ({\mathcal D}_{i}, \pi_{G})$ induces an automorphism $f_{G}: P \ra P$ of the Tanaka $\{ e\}$-structure
of $P$, by $f_{G} (u):= f_{*}\circ u$, for any
graded frame $u: \gm \ra \mathrm{gr} (T_{p}M)$ which belongs to
$P$,
and that $f \mapsto f_{G}$ is an isomorphism  betweeen these Tanaka structure automorphism groups.
Next, one notices (from definitions) that the automorphisms of the
Tanaka $\{ e\}$-structure of $P$ coincide with
the automorphisms of the
$G$-structure $\tilde{\pi}^{1}: \tilde{P}^{1} \ra P$.

It remains to prove that $\mathrm{Aut} (\tilde{\pi}^{n})$ is
isomorphic to $\mathrm{Aut} (\tilde{\pi}^{n+1})$, for any $n\geq 1.$ Any
$\bar{f}^{(n-1)}\in \mathrm{Aut} (\tilde{\pi}^{n})$ induces a map
$f_{\tilde{P}^{n}} : \tilde{P}^{n} \ra \tilde{P}^{n}$, defined by $f_{\tilde{P}^{n}} (F_{H^{n}} )
:=  (\bar{f}^{(n-1)})_{*} \circ F_{H^{n}}$, for any $F_{H^{n}}\in  \tilde{P}^{n}$.
The map
$f_{\tilde{P}^{n}}$  commutes with the action of $G^{n} GL_{n+1} (\gm_{n-1})$
(hence, also with the action of $GL_{n+1} (\gm_{n-1})$)
 on $\tilde{P}^{n}$  and induces a map $\bar{f}^{(n)}: \bar{P}^{(n)} \ra \bar{P}^{(n)}$ which belongs to
 $ \mathrm{Aut} (\tilde{\pi}^{n+1})$ (easy check).
For the converse, let
$\bar{f}^{(n)}\in \mathrm{Aut} (\tilde{\pi}^{n+1})$, i.e. $\bar{f}^{(n)}: \bar{P}^{(n)} \ra \bar{P}^{(n)}$
is a diffeomorphism,  such that,
for any frame $F_{H^{n+1}} : \gm_{n} \ra T_{\bar{H}^{n}} \bar{P}^{(n)}$ which belongs to
$\tilde{P}^{n+1}$, $(\bar{f}^{(n)})_{*} \circ F_{H^{n+1}}: \gm_{n} \ra
T_{\bar{f}^{(n)}(\bar{H}^{n})} \bar{P}^{(n)}$ also belongs to $\tilde{P}^{n+1}.$
Since the frames from $\tilde{P}^{n+1}$ are filtration preserving,
both $F_{H^{n+1}}$ and $(\bar{f}^{(n)})_{*} \circ F_{H^{n+1}}$, therefore also
$\bar{f}^{(n)}$, are filtration preserving.
Since the frames from $\tilde{P}^{n+1}$, restricted to $\gg^{n}$, coincide with the vertical parallelism
of $\bar{\pi}^{(n)}: \bar{P}^{(n)} \ra \bar{P}^{(n-1)}$, we obtain that
$(\bar{f}^{(n)})_{*} (\xi^{v})^{\bar{P}^{(n)}} = (\xi^{v})^{\bar{P}^{(n)}}$, for any $v\in \gg^{n}.$
Therefore, there is $\bar{f}^{(n-1)}: \bar{P}^{(n-1)} \ra \bar{P}^{(n-1)}$ such that
$\bar{\pi}^{(n)} \circ \bar{f}^{(n)} = \bar{f}^{(n-1)} \circ \bar{\pi}^{(n)}.$
We check that $\bar{f}^{(n-1)}$ induces $\bar{f}^{(n)}$. For this, we use: for any
$x\in (\gm_{n-1})^{i}$, 
\begin{align}
\nonumber&\mathrm{pr}^{i}_{(n+1)} (\bar{\pi}^{(n)})_{*}  F^{i}_{H^{n+1}}(x) = F^{i}_{\bar{H}^{n}}(x),\\
\label{nat-dif}&\mathrm{pr}^{i}_{(n+1)}
(\bar{\pi}^{(n)})_{*} (\bar{f}^{(n)})_{*} F^{i}_{H^{n+1}}(x)  = F^{i}_{\bar{f}^{(n)}
(\bar{H}^{n})}(x),
\end{align}
where $\bar{H}^{n-1} = \bar{\pi}^{(n)} (\bar{H}^{n}).$
(Relations (\ref{nat-dif}) follow from   $F_{H^{n+1}}, (\bar{f}^{(n)})_{*} \circ F_{H^{n+1}}\in \tilde{P}^{n+1}$ and
Lemma \ref{ajut-p2}). Since  $\bar{\pi}^{(n)} \circ \bar{f}^{(n)} = \bar{f}^{(n-1)} \circ \bar{\pi}^{(n)}$ and
$\bar{f}^{(n)}$, $\bar{f}^{(n-1)}$ are filtration preserving, we obtain from relations (\ref{nat-dif}) that $ F_{\bar{f}^{(n)} (\bar{H}^{n})}=
(\bar{f}^{(n-1)})_{*} \circ F_{\bar{H}^{n}} $, i.e.
$\bar{f}^{(n)}$ is induced by $\bar{f}^{(n-1)}$, as required.
It is easy to see that $\bar{f}^{(n-1)}\in \mathrm{Aut}( \tilde{\pi}^{n})$.
\end{proof}

\begin{prop}\label{ad-prop-2} Let $({\mathcal D}_{i}, \pi_{G})$  be a Tanaka $G$-structure
of type $\gm = \sum_{i=-k}^{-1}\gm^{i}$ and  finite order $\bar{l}$ and $F^{\mathrm{can}}$
the canonical frame of $\bar{P}^{(\bar{l})}$. Then
$\mathrm{Aut} ({\mathcal D}_{i}, \pi_{G})$
is isomorphic to   $\mathrm{Aut}(\bar{P}^{(\bar{l})},  F^{\mathrm{can}})$. It is a Lie group
with $\mathrm{dim}\mathrm{Aut} ({\mathcal D}_{i}, \pi_{G})\leq
\mathrm{dim} (M) + \sum_{i=0}^{\bar{l}  }\mathrm{dim}( \gg^{i}).$
\end{prop}

\begin{proof}
The argument from Proposition \ref{ad-prop} provides a canonical frame (an absolute parallelism)
$(F^{\mathrm{can}})^{\prime}$
on any prolongation $\bar{P}^{(\bar{l}^{\prime})}$ (with $\bar{l}^{\prime}\geq \bar{l}$),
isomorphic to $F^{\mathrm{can}}$ by means of the map $\bar{\pi}^{(\bar{l}^{\prime} , \bar{l}+1)}.$
Let $\bar{l}^{\prime}\geq \bar{l}$ sufficiently large such that
$\pi^{{\bar{l}}^{\prime}}: P^{\bar{l}^{\prime}} \ra \bar{P}^{( \bar{l}^{\prime} -1)}$ is
an $\{ e\}$-structure (recall Definition \ref{pi-n-t}
for $\pi^{\bar{l}^{\prime}}$).
Then, for any $s\geq \bar{l}^{\prime}$, $P^{s} = \tilde{P}^{s} = \bar{P}^{(s)}$ and
$\pi^{s} = \tilde{\pi}^{s} = \bar{\pi}^{(s)}$ is an $\{ e\}$-structure. Any
$\bar{H}^{s}\in \bar{P}^{(s)}$ ($s\geq \bar{l}^{\prime}$)
is a frame $F_{\bar{H}^{s}} : \gm_{\bar{l}} \ra T_{\bar{\pi}^{(s)}(\bar{H}^{{s} })} \bar{P}^{(s-1)}$.
By construction of the prolongations,  the  preimage  $({\pi}^{s+1})^{-1} (\bar{H}^{s})\in P^{s+1}$
is the unique
frame $F_{({\pi}^{s+1})^{-1} (\bar{H}^{s})} : \gm_{\bar{l}} \ra T_{\bar{H}^{s}}\bar{P}^{(s)}$ given by
\begin{equation}\label{pi-prop-fin}
(\bar{\pi}^{(s)})_{*} \circ F_{( {\pi}^{s+1})^{-1} (\bar{H}^{s})}= F_{\bar{H}^{s}},\quad s\geq \bar{l}^{\prime} .
\end{equation}
Consider now  the canonical frame $(F^{\mathrm{can}})^{\prime}$ of $\bar{P}^{(\bar{l}^{\prime})}$.
From the proof of Proposition \ref{ad-prop}, it  is defined by
\begin{equation}\label{pi-prop-fin-1}
(F^{\mathrm{can}})^{\prime}_{\bar{H}^{\bar{l}^{\prime}} } =  (\bar{\pi}^{(k+ \bar{l}^{\prime} ,\bar{l}^{\prime} + 1)})_{*} \circ F_{
(\bar{\pi}^{(k+\bar{l}^{\prime} +1, \bar{l}^{\prime} +1)})^{-1}(\bar{H}^{\bar{l}^{\prime} })}.
\end{equation}
We will show that $(F^{\mathrm{can}})^{\prime}$ is the $\{ e\}$-structure $\pi^{\bar{l}^{\prime} +1}$
of $\bar{P}^{(\bar{l}^{\prime})}.$
From relation  (\ref{pi-prop-fin}), we need to check that
$ (\bar{\pi}^{(\bar{l}^{\prime})})_{*} \circ (F^{\mathrm{can}})^{\prime}_{\bar{H}^{\bar{l}^{\prime}}}
= F_{\bar{H}^{\bar{l}^{\prime}}}$, for any $\bar{H}^{\bar{l}^{\prime}}\in
\bar{P}^{(\bar{l}^{\prime})}$, or, using
relation (\ref{pi-prop-fin-1}), $( \bar{\pi}^{(k+\bar{l}^{\prime}, \bar{l}^{\prime})})_{*} \circ F_{\bar{H}^{k+\bar{l}^{\prime} +1}}
= F_{\bar{\pi}^{(k+\bar{l}^{\prime} +1, \bar{l}^{\prime} +1)}(\bar{H}^{k+ \bar{l}^{\prime}+1})}$,
for any $\bar{H}^{k+ \bar{l}^{\prime} +1}\in \bar{P}^{k+\bar{l}^{\prime} +1}.$
The latter relation  follows easily from (\ref{pi-prop-fin}).
We  obtain that $(F^{\mathrm{can}})^{\prime}$ coincides with the absolute parallelism $\pi^{\bar{l}^{\prime} +1}$
on $\bar{P}^{(\bar{l}^{\prime})}$.
From Proposition \ref{ad-prop-1},
$\mathrm{Aut} (\bar{P}^{(\bar{l}^{\prime})}, (F^{\mathrm{can}})^{\prime})$ (or
 $\mathrm{Aut} (\bar{P}^{(\bar{l})}, F^{\mathrm{can}})$)
is isomorphic to  $\mathrm{Aut} (\mathcal D_{i}, \pi_{G})$.
From Kobayashi theorem   (see
Theorem 3.2 of \cite{kobayashi}, p. 15),  these groups are Lie groups of dimension at most
$\mathrm{dim} (M) + \sum_{i=0}^{\bar{l}  }\mathrm{dim}( \gg^{i})$.\end{proof}

\end{document}